\documentclass[12pt]{article} 
\usepackage{latexsym}
\usepackage{soul}

\usepackage[margin=2.5cm]{geometry}
\usepackage[english]{babel}
\usepackage{hyperref}
\usepackage{amsmath,mathtools}
\usepackage{amsthm}
\usepackage{amssymb}
\usepackage{dsfont}
\usepackage{ucs}
\usepackage{enumerate}
\usepackage[utf8]{inputenc}
\usepackage[T1]{fontenc}
\usepackage{xcolor}
\usepackage{tikz}

\usepackage[draft,inline,nomargin,index]{fixme}

\fxsetup{theme=color,mode=multiuser}
\FXRegisterAuthor{ac}{aac}{\color{purple}AC}
\FXRegisterAuthor{mj}{amj}{\color{red}MJ}
\FXRegisterAuthor{rt}{art}{\color{orange}RT}
\FXRegisterAuthor{flm}{aflm}{\color{blue}FLM}

\DeclareMathOperator{\rng}{\mathrm{rng}}
\DeclareMathOperator{\supp}{\mathrm{supp}}

\newcommand{\MAlg}{\mathrm{MAlg}}
\newcommand{\Aut}{\mathrm{Aut}}
\newcommand{\ERG}{\mathrm{ERG}}
\newcommand{\WMIX}{\mathrm{WMIX}}
\newcommand{\APER}{\mathrm{APER}}

\newcommand{\R}{\mathbb R}
\newcommand{\N}{\mathbb N}
\newcommand{\Q}{\mathbb Q}
\newcommand{\Z}{\mathbb Z}
\newcommand{\C}{\mathbb C}
\newcommand{\LL}{\mathrm L}

\newcommand{\abs}[1]{\left\lvert #1\right\rvert}

\newcommand{\dom}{\mathrm{dom}}

\newcommand{\inv}{^{-1}}

\renewcommand{\leq}{\leqslant}
\renewcommand{\geq}{\geqslant}

\renewcommand{\d}{\normalfont\textsf{d}}

\newcommand{\into}{\hookrightarrow}

\newtheorem{thm}{Theorem}[section]
\newtheorem{cor}[thm]{Corollary}

\newtheorem{lem}[thm]{Lemma}
\newtheorem{prop}[thm]{Proposition}

\theoremstyle{definition}
\newtheorem{claim}{Claim}
\newtheorem*{claim*}{Claim}

\newenvironment{cproof}{\begin{proof}[Proof of the 
		claim]}{\end{proof}}

\newtheorem{qu}[thm]{Question}

\newtheorem{df}[thm]{Definition}
\newtheorem{rmq}[thm]{Remark}
\newtheorem{example}[thm]{Example}

\newtheorem*{convention}{Convention}

\newcommand{\ph}{\varphi}
\newcommand{\eps}{\varepsilon}

\renewcommand{\d}{\boldsymbol{d}}
\newcommand{\id}{\mathrm{id}}

\newcommand{\defin}[1]{\textbf{\textit{#1}}}

\title{Belinskaya's theorem is optimal}
\author{Alessandro Carderi, Matthieu Joseph, François Le Maître, Romain Tessera}

\begin{document}
	
	\newcommand{\Addresses}{{
			\bigskip
			\footnotesize
			
			\noindent A.~Carderi, \textsc{Institut für Algebra und Geometrie, KIT, 76128 Karlsruhe, GERMANY}\par\nopagebreak\noindent
			\textit{E-mail address: }\texttt{alessandro.carderi@kit.edu}
			
			\medskip
			
			\noindent M.~Joseph, \textsc{Université de Lyon, ENS de Lyon, Unité de Mathématiques Pures et Appliquées, 46,
				allée d’Italie 69364 Lyon Cedex 07, FRANCE}\par\nopagebreak\noindent
			\textit{E-mail address: }\texttt{matthieu.joseph@ens-lyon.fr}
			
			\medskip
			
			\noindent F.~Le Maître, \textsc{Établissement public expérimental – Décret N°2019-209 du 20 mars 2019, Sorbonne Université, CNRS, Institut de Mathématiques
				de Jussieu-Paris Rive Gauche, F-75013 Paris, FRANCE}\par\nopagebreak\noindent
			\textit{E-mail address: }\texttt{francois.le-maitre@imj-prg.fr}
			
			\medskip
			
			\noindent R.~Tessera, \textsc{CNRS, Établissement public expérimental – Décret N°2019-209 du 20 mars 2019, Sorbonne Université,  Institut de Mathématiques
				de Jussieu-Paris Rive Gauche, F-75013 Paris, FRANCE}\par\nopagebreak\noindent
			\textit{E-mail address: }\texttt{romatessera@gmail.com}

			\medskip

	}}

	\maketitle
	\abstract{
	Belinskaya's theorem states that given an ergodic measure-preserving transformation, any other transformation with the same orbits and an $\mathrm{L}^1$ cocycle must be flip-conjugate to it. Our main result shows that this theorem is optimal: for all $p<1$ the integrability condition on the cocycle cannot be relaxed to being in $\mathrm{L}^p$.
	This also allows us to answer a question of Kerr and Li: for ergodic measure-preserving transformations, Shannon orbit equivalence  doesn't boil down to flip-conjugacy.}

	\tableofcontents
	
	\listoffixmes
	\section{Introduction}
	
	Given two ergodic measure-preserving (invertible) transformations $T_1,T_2$ of a standard probability space $(X,\mu)$, the \emph{conjugacy problem} asks whether there is a third measure-preserving invertible transformation $S$ such that $ST_1=T_2S$. 
	Although the conjugacy problem is intractable in full generality, various invariants have been devised over the years. Two of the most important ones are the \emph{spectrum} and the \emph{dynamical entropy}. 
	The first completely classifies compact transformations \cite{halmosOperatorMethodsClassical1942}, while the second completely classifies Bernoulli shifts \cite{sinaiNotionEntropyDynamical1959,ornsteinBernoulliShiftsSame1970}.
	
	In this paper, we are interested in natural weakenings of the conjugacy problem obtained through the notion of \emph{orbit equivalence}.
	Two measure-preserving transformations $T_1$, $T_2$ are \defin{orbit equivalent} if there is a measure-preserving transformation $S$ such that $ST_1S\inv$ and $T_2$ have the same orbits (such an $S$ is called an orbit equivalence between $T_1$ and $T_2$).
	A stunning theorem of Dye states that all ergodic measure-preserving transformations of a standard probability space are orbit equivalent \cite{dyeGroupsMeasurePreserving1959}, so  orbit equivalence for measure-preserving ergodic transformations is a weakening of conjugacy which turns out to be the trivial relation. 
	
	In order to circumvent this indistinguishability, we will compare orbit equivalences between measure-preserving transformations in a quantitative way. This fits into the emerging field of \textit{quantitative orbit equivalence} for group actions. One of its tacit aims is to capture meaningful geometric invariants, such as F\o{}lner functions \cite{delabieQuantitativeMeasureEquivalence2020}, growth rates \cite{austinIntegrableMeasureEquivalence2016}, etc., or ergodic theoretic invariants, such as dynamical entropy \cite{austinBehaviourEntropyBounded2016}. 

	In our setup of measure-preserving transformations, quantifications will be imposed on \emph{orbit equivalence cocycles}. Given an orbit equivalence $S$ between two ergodic measure-preserving transformations $T_1$ and $T_2$, the orbit equivalence cocycles $c_1,c_2:X\to\Z$ are the maps uniquely defined by the following equation: for all $x\in X$
	
	\begin{equation}\label{eq: c1 et c2}
		ST_1(x)=T_2^{c_2(x)}S(x)\text{ and }T_2S(x)=ST_1^{c_1(x)}(x).
	\end{equation}
	
	Belinskaya's theorem is probably the first result on quantitative orbit equivalence.  
	In the litterature, it is often stated as a symmetric result on integrable orbit equivalence of ergodic measure-preserving transformations. 
	However, her result is asymmetric and can be stated as follows.
	
	\begin{thm}[{Belinskaya \cite{belinskayaPartitionsLebesguespace1968}}]\label{thmIntro:BelinskayaAssym}
		Let $T_1$ and $T_2$ be two ergodic measure-preserving transformations, let $S$ be an orbit equivalence between them and suppose that the previously defined cocycle $c_1$ is integrable, 
		i.e.\ \[\int_X\abs{c_1(x)}d\mu<+\infty.\] Then $T_1$ and $T_2$ are flip-conjugate: either $T_1$ is conjugate to $T_2$ or $T_1\inv$ is conjugate to $T_2$.
	\end{thm}
	
	It is natural to wonder whether Belinskaya's theorem remains valid if one weakens the integrability assumptions. For example, one could ask that one of the orbit equivalence cocycle belongs to $\LL^p(X,\mu)$ for some $p\in (0,1)$.
	
    We will consider more general integrability assumptions. Given a function $\ph\colon\R_+\to\R_+$, we say that a measurable integer-valued function $f$ is $\ph$-\defin{integrable} if \[\int_X\varphi(\abs{f(x)})d\mu<+\infty.\]
	Our first main result concerns orbit equivalence of measure-preserving transformations for which one of the orbit equivalence cocycles is $\ph$-integrable, for some \defin{sublinear} function $\ph$, that is satisfying $\lim_{t\to+\infty} \varphi(t)/t=0$. This is for example the case for $\ph(t)=t^p$ where $p\in (0,1)$. With this integrability condition, the conclusion of Belinskaya's theorem does not hold.
	
	\begin{thm}[{see Theorem \ref{thm:genericcounterexamples}}]\label{thmIntro:belinskaya opti asym}
		Let $\varphi:\R_+\to \R_+$ be a sublinear function and $T_1$ be an ergodic measure-preserving transformation. Then there is an ergodic measure-preserving transformation $T_2$ and an orbit equivalence $S$ between $T_1$ and $T_2$ such that the associated cocycle $c_1$ is $\varphi$-integrable but the transformations $T_1$ and $T_2$ are not flip-conjugate.
	\end{thm}

	The fact that the hypotheses on $\ph$ are fairly weak gives us much freedom. For example, the above theorem even implies that Belinskaya's theorem does not hold if we assume that one of the two orbit equivalence cocycles belongs to $\LL^p(X,\mu)$ for \emph{all} $p\in (0,1)$. Indeed if we consider for instance the sublinear function $\varphi(t)=t/\ln(t+1)$, then $\varphi$-integrability implies being in $\LL^p(X,\mu)$ for all $p<1$.

	A symmetric way to strengthen Theorem \ref{thmIntro:belinskaya opti asym} involves the concept of $\ph$-integrable orbit equivalence. We say that two measure-preserving transformations are \defin{$\varphi$-integrably orbit equivalent} if there is an orbit equivalence $S$ such that \emph{both} orbit equivalence cocycles $c_1$ and $c_2$ are $\varphi$-integrable. In this context, we obtain a similar conclusion to Theorem \ref{thmIntro:belinskaya opti asym}, but we have to make one additional assumption on $T_1$.

	\begin{thm}[{see Corollary \ref{cor.belinskayaoptimal}}]\label{thmIntro:belinskaya opti sym}
		Let $\ph :\R_+\to\R_+$ be a sublinear function. Let $T_1$ be an ergodic measure-preserving transformation and assume that $(T_1)^n$ is ergodic for some $n\geq 2$. Then there is another ergodic measure-preserving transformation $T_2$ such that $T_1$ and $T_2$ are $\varphi$-integrably orbit equivalent but not flip-conjugate.
	\end{thm}
	Concrete examples of transformations to which this theorem applies are Bernoulli shifts, irrational rotations on the circle and the $m$-odometer for any integer $m$.
	One can show that the only ergodic measure-preserving transformations that are \emph{not} covered by this theorem are the ones that factor onto the universal odometer, that is, the transformation $t\mapsto t+1$ on the profinite completion $\widehat{\Z}$.
	
	Let us point out that the proof of Theorem \ref{thmIntro:belinskaya opti asym} uses Theorem \ref{thmIntro:belinskaya opti sym}, so the two results are not independent. As we will explain later, Theorem \ref{thmIntro:belinskaya opti asym} also depends on the Baire category theorem. 
	
	However, Theorem \ref{thmIntro:belinskaya opti sym} is somewhat more explicit. 
The starting point of the proof is the following simple construction, which was already used in \cite[Thm.~4.8]{lemaitremeasurableanaloguesmall2018}. 
Let us fix an ergodic transformation $T_1$ with $(T_1)^n$ ergodic. Suppose we have a periodic transformation $P$  all of whose orbits have cardinality $n$ and are contained in those of $T_1$.
Consider the transformation $T_2$ obtained by composing $P$ with the transformation induced by $T_1$ on a fundamental domain of $P$. Then $T_1$ and $T_2$ have the same orbits. However, $(T_2)^n$ is not ergodic and thus $T_1$ and $T_2$ are not flip-conjugate.  The heart of our proof is therefore to construct $P$ so that the orbit equivalence cocycles between $T_1$ and $T_2$ satisfy the required integrability conditions.

	For many concrete measure-preserving transformations, explicit examples of such periodic transformations $P$ with specific integrability conditions can be obtained. We will give details in the case of the Bernoulli shift, see Example \ref{ex: bernoulli involutions}.

	\paragraph*{Shannon orbit equivalence and dynamical entropy}  A remarkable consequence of Theorem \ref{thmIntro:belinskaya opti sym} can be stated in the context of \emph{Shannon orbit equivalence}, as defined by Kerr and Li \cite{kerrEntropyShannonOrbit2019}. Two measure-preserving transformations are Shannon orbit equivalent if there exists an orbit equivalence between them whose orbit equivalence cocycles $c_1$ and $c_2$ both have finite Shannon entropy. In \cite{austinBehaviourEntropyBounded2016}, it was implicitely shown that, among actions of finitely generated amenable groups which are not virtually cyclic, Shannon orbit equivalence preserves dynamical entropy. This was then generalized in \cite{kerrEntropyShannonOrbit2019} to countable amenable groups that are neither locally finite nor virtually cyclic, as well as some nonamenable groups. Kerr and Li implicitly asked whether dynamical entropy is an invariant of Shannon orbit equivalence for measure-preserving transformations and wondered whether Shannon orbit equivalence could actually directly imply flip-conjugacy. They answered positively the first part of this question in \cite{kerr2022entropy}. Here, we show that Shannon orbit equivalence doesn't boil down to flip-conjugacy.


	\begin{thm}[{see Theorem \ref{thm.reponsekerrli}}]\label{thmIntro.question kerr and li}
		Let $T_1\in\Aut(X,\mu)$ be an ergodic transformation and assume that $(T_1)^n$ is ergodic for some $n\geq 2$. Then there exists $T_2\in\Aut(X,\mu)$ such that $T_1$ and $T_2$ are Shannon orbit equivalent but not flip-conjugate. 
	\end{thm}
	
	The above theorem is obtained by applying Theorem \ref{thmIntro:belinskaya opti sym} with any sublinear function $\ph$ such that $\ln(1+t)=O(\varphi(t))$. Indeed, for any such function, $\ph$-integrable orbit equivalence implies Shannon orbit equivalence, see Theorem \ref{thm.intOEimpliesShannonOE}.
	
	We also observe that Shannon orbit equivalence preserves finiteness of dynamical entropy, see Proposition \ref{prop.finiteness entropy preserved}. This is now subsumed by a recent preprint of Kerr and Li who proved that the dynamical entropy is preserved under Shannon orbit equivalence \cite{kerr2022entropy}.

	\begin{qu}\label{qu: phi oe and entropy}
		For which unbounded sublinear metric-compatible functions $\varphi$ is it true that dynamical entropy is an invariant of $\ph$-integrable orbit equivalence?
	\end{qu}

By the above discussion, we already know that this holds for all $\varphi$ such that \mbox{$\ln(1+t)=O(\varphi(t))$}. On the other hand, using Dye's theorem, it is not hard to see that any two ergodic measure-preserving transformations are $\varphi$-integrably orbit equivalent for \emph{some} sublinear unbounded function $\varphi$ (cf.~the proof of \cite[Prop.~4.24]{delabieQuantitativeMeasureEquivalence2020}). So not every sublinear unbounded function satisfies the condition of the question. 
	
	\paragraph*{$\ph$-integrable full groups}
	
	The proof of both our main results will make crucial use of the notion of $\ph$-\textit{integrable full group}. Whenever $T$ is an ergodic measure preserving transformation of the probability space $(X,\mu)$, Dye defined a Polish group $[T]$, called the full group of $T$. This group is by definition the set of all measure-preserving transformations $U$ of $(X,\mu)$ whose orbits are contained in $T$-orbits. More precisely, $U\in [T]$ if there is a function $c_U$, called the $T$-\defin{cocycle of }$U$, such that $U(x)=T^{c_U(x)}(x)$ for all $x\in X$. The above stated theorem of Dye, that all ergodic transformations are obit equivalent, was originally stated in terms of full groups: whenever $T_1$ and $T_2$ are ergodic transformations, the full groups $[T_1]$ and $[T_2]$ are conjugate. 
	
	In our context, once $\ph$ is fixed, the reasonable analogue of the full group associated to this integrability condition would be  the set of transformations $U\in [T]$ such that the cocycle $c_U$ is $\ph$-integrable. However, for this set to be a subgroup of $[T]$, we will have to impose a mild restriction on $\ph$. We say that $\varphi:\R_+\to\R_+$ is a  \defin{metric-compatible function} if 
	\begin{itemize}
		\item (subadditivity) for all $s,t\in\R_+$, $\varphi(s+t)\leq \varphi(s)+\varphi(t)$.
		\item (separation) $\varphi(0)=0$ and $\varphi(t)>0$ for all $t>0$.
		\item (monotonicity) $\varphi$ is a non-decreasing function.
	\end{itemize}
	
	The name metric-compatible comes from the observation that whenever $d$ is a metric and $\ph$ a metric-compatible function, then $\ph\circ d$ is also a metric. The following theorem is a combination of Lemma \ref{lem.distanceonphifullgroup} and Theorem \ref{thm.polishgrouptopology}.
	
	\begin{thm}
		Let $\ph$ be a metric-compatible function and let $T$ be a measure preserving transformation of the probability space $(X,\mu)$. Then the set \[[T]_\ph\coloneqq \left\lbrace U\in [T]\colon \int_X \ph(\lvert c_U(x)\rvert)d\mu<+\infty\right\rbrace\]
		is a group. Moreover the function \[\d_{\ph,T}(U,V)\coloneqq\int_X\ph(\lvert c_U(x)-c_V(x)\rvert)d\mu\]
		is a complete, right-invariant and separable metric on $[T]_\ph$ whose induced topology is a group-topology. In particular $([T]_\ph,\d_{\ph,T})$ is a Polish group. 
	\end{thm}
	
	It turns out that any sublinear function is dominated by a sublinear metric-compatible function, see Lemma \ref{lem.free metric-compatible}. This will allows us to reduce the proof of Theorems \ref{thmIntro:belinskaya opti asym} and \ref{thmIntro:belinskaya opti sym} to the case where $\ph$ is metric-compatible and thereby to exploit the group structure of $[T]_\ph$.
	
	\paragraph*{Genericity of weakly mixing} Let us come back to Theorem \ref{thmIntro:belinskaya opti asym}. As the conclusions of Theorem \ref{thmIntro:belinskaya opti sym} are stronger, we just need to show Theorem \ref{thmIntro:belinskaya opti asym} whenever $(T_1)^n$ is non-ergodic for all $n\geq 2$. Observe that this condition is incompatible with the notion of \textit{weakly mixing}, as all the powers of any weakly mixing transformation are ergodic. Therefore our strategy is to provide for every ergodic transformation $T_1$ a weakly mixing transformation $T_2$ which has the same orbits as $T_1$ and whose $T_1$-cocycle is $\ph$-integrable. 
    We do not have any constructive argument for this and we proceed through the Baire category theorem, 
    inspired by a similar result in the context of full groups of ergodic measure-preserving equivalence relations (see \cite[Thm.~3.6]{kechrisGlobalaspectsergodic2010}). 
	
	\begin{thm}[{see Theorem \ref{thm.gdelta dense weakly mixing}}]\label{thmIntro.gdelta dense weakly mixing}
		Let $\ph$ be a sublinear metric-compatible function and let $T\in\Aut(X,\mu)$ be an aperiodic element. Then the set of all measure-preserving transformations in $[T]_\ph$ which are weakly mixing and have the same orbits as $T$ is a dense $G_\delta$ set in the Polish space of aperiodic transformations of $[T]_\ph$.
	\end{thm}
	
	Besides the Baire category theorem, there are two other main ingredients in the proof of Theorem \ref{thmIntro.gdelta dense weakly mixing}. One is a result of Conze \cite{conzeEquationsFonctionnellesSystemes1972} which claims that starting from any ergodic mesure-preserving transformation, the first return map to a \textit{generic} measurable subset gives rise to a weakly mixing transformation. The second is a sublinear ergodic theorem which may be of independent interest.   
	\begin{thm}[{see Theorem \ref{thm.subadditivephifullgroup}}]\label{thmIntro.subadditivephifullgroup}
		Let $\ph:\R_+\to\R_+$ be a sublinear metric-compatible function. Let $U\in\Aut(X,\mu)$ and $f\colon X\rightarrow \mathbb C$ measurable such that $\ph(\lvert f\rvert) \in \LL^1(X,\mu)$. Then for almost every $x\in X$
		\[\lim_n \frac 1 n \ph\left(\left|\sum_{k=0}^{n-1} f(U^k(x))\right|\right)=0.\]
		The convergence also holds in $\LL^1$, that is
		\[\lim_n\int_X \frac 1 n\ph\left(\left|\sum_{k=0}^{n-1} f(U^k(x))\right|\right) d\mu=0.\]
	\end{thm}
	
	\paragraph*{Outline of the paper} In Section \ref{sec.preliminaries}, after a few preliminaries, we present the framework and establish basic properties of $\varphi$-integrable full groups. In Section \ref{sec.belinskaya}, we explain our construction of periodic transformations in $\ph$-integrable full groups and use it to prove Theorem \ref{thmIntro:belinskaya opti asym}. In Section \ref{sec.weaklymixing}, we first prove that $\ph$-integrable full groups are Polish groups. We then use the Baire category theorem and prove that weakly mixing elements are generic in the set of aperiodic elements in $[T]_\ph$. Combining this with Theorem \ref{thmIntro:belinskaya opti asym}, we finally prove Theorem \ref{thmIntro:belinskaya opti sym}. In the appendix, we present a proof of Belinskaya's theorem which is due to Katznelson and is not publicly available to our knowledge. 

    \paragraph*{Acknowledgments.} We thank the referee for many remarks and corrections and for pointing out a mistake in an earlier version of Theorem \ref{thm.subadditivephifullgroup}. We also thank Fabien Hoareau for pointing out several inaccuracies in the appendix.

	\section{Quantitative orbit equivalence and full groups}\label{sec.preliminaries}
	
	\subsection{Preliminaries}
	
	\begin{sloppypar}
		Throughout the paper, $(X,\mu)$ will denote a standard probability space without atoms. Recall that such spaces are measure isomorphic to the interval $[0,1]$ equipped with the Lebesgue measure. A bimeasurable bijection $T \colon X\to X$ is a \defin{measure-preserving transformation} of $(X,\mu)$ if for all measurable sets $A\subseteq X$, one has $\mu(T^{-1}(A))=\mu(A)$. We denote by $\Aut(X,\mu)$ the group of all measure-preserving transformations of $(X,\mu)$, two such transformations being identified if they coincide on a conull set. The group $\Aut(X,\mu)$ will be equipped with the \defin{uniform metric} $d_u$ defined by 
		\[d_u(T_1,T_2)\coloneqq \mu(\lbrace x\in X\colon T_1(x)\neq T_2(x)\rbrace).\] 
		This metric is bi-invariant and complete \cite[p.~73]{halmosLecturesErgodicTheory2017}.
	\end{sloppypar}
	
	\begin{rmq}\label{remk.no null sets}
        
         We will often use implicitely the fact that two measure-preserving transformations which belongs to the same class in $\Aut(X,\mu)$ agree on a invariant conull set. 
	\end{rmq}

	The \defin{support} of a measure-preserving transformation $T\in\Aut(X,\mu)$ is the measurable set $\supp(T)\coloneqq\{x\in X\colon T(x)\neq x\}$.
	
	A measure-preserving transformation $T\in\Aut(X,\mu)$ is \defin{periodic} if the $T$-orbit of almost every $x\in X$ is finite.
	A \defin{fundamental domain} of a periodic transformation $T\in\Aut(X,\mu)$ is a measurable subset $D\subseteq X$ which intersects almost every $T$-orbit at exactly one point. Every periodic transformation admits such a fundamental domain, as can be seen by fixing a Borel linear order $<$ on $X$ and taking for $D$ the set of $<$-least points in each orbit of the transformation. A measure-preserving transformation $T\in\Aut(X,\mu)$ is \defin{aperiodic} if the $T$-orbit of almost every $x\in X$ is infinite. It is \defin{ergodic} if every $T$-invariant measurable set is either null or conull.
	
	The \defin{full group} of a measure-preserving transformation $T$ is the group 
    \[[T]\coloneqq\left\lbrace U\in\Aut(X,\mu)\colon \text{for a.e. } x\in X,\ \exists n\in\mathbb Z\text{ s.t. }U(x)= T^n(x)\right\rbrace.\]
	
	\begin{rmq}

        By Remark \ref{remk.no null sets}, $U\in [T]$ if and only if for a.e.~$x\in X$, the $U$-orbit of $x$ is contained in the $T$-orbit of $x$. 
	\end{rmq}
	
	Two measure-preserving transformations $T_1,T_2\in\Aut(X,\mu)$ \defin{have the same orbits} if for almost every $x\in X$, the $T_1$-orbit of $x$ coincides with the $T_2$-orbit of $x$. By the above remark, this is equivalent to following condition: $T_1\in[T_2]$ and $T_2\in[T_1]$. We say that two measure-preserving transformations $T_1,T_2\in\Aut(X,\mu)$ are \defin{orbit equivalent} if there exists $S\in\Aut(X,\mu)$ such that $ST_1S^{-1}$ and $T_2$ have the same orbits, that is $ST_1S^{-1}\in[T_2]$ and $T_2\in[ST_1S^{-1}]$.
	
	Fix an aperiodic transformation $T\in\Aut(X,\mu)$. Any $U\in [T]$ is completely determined by its \defin{$T$-cocycle}, defined as the unique function $c_U\colon X\to\Z$ satisfying the equation $U(x)=T^{c_U(x)}(x)$ for all $x\in X$. The $T$-cocycle satisfies the so-called cocycle identity: given $U,V\in [T]$, we have 
	\begin{equation}\label{eq:cocycle identity}c_{UV}(x)=c_U(V(x))+c_V(x)\text{ for all }x\in X.\end{equation}

	Let $T\in\Aut(X,\mu)$ and $A\subseteq X$ be a measurable subset. The \defin{first return time} of $T$ to $A$ is the map $n_{T,A}\colon A\rightarrow \mathbb{N}^*$ defined by \[n_{T,A}(x)\coloneqq \min \{n\in\N^*\colon T^n(x)\in A\}.\]
	This function is well-defined up to measure zero by Poincaré's recurrence theorem.
	For convenience, we extend $n_{T,A}$ to all $X$, setting it to be $0$ on $X\setminus A$.
	Kac's lemma \cite{kacnotionrecurrencediscrete1947} yields the following inequality
	\begin{equation}\label{equation.kac}
		\int_X n_{T,A}(x)d\mu\leq 1.
	\end{equation}
	The \defin{first return map} of $T$ with respect to $A$ is the transformation $T_A\in [T]\leq \Aut(X,\mu)$ defined by \[T_A(x)\coloneqq T^{n_{T,A}(x)}(x).\]
	By definition, we have $\supp(T_A)=A$ and $x,y\in A$ are in the same $T$-orbit if and only if they are in the same $T_A$-orbit. Whenever $T$ is aperiodic, the first return time $n_{T,A}$ coincides with the $T$-cocycle $c_{T_A}$ of $T_A$.
	\begin{lem}\label{lem.sameorbits}
		Let $T\in\Aut(X,\mu)$, let $P\in\Aut(X,\mu)$ be a periodic transformation and $D$ a fundamental domain of $P$. 
		Let $U\coloneqq T_DP$. Then the following are true.
		\begin{enumerate}[(i)]
			\item $U_D=T_D$.\label{item.samefirstreturnmap} 
			\item\label{item.returntime} If $x\in D$ and $n(x)$ is the cardinality of the $P$-orbit of $x$, then  $n_{U,D}(x)=n(x)$.
			\item If $P\in [T]$, then $T$ and $U$ have the same orbits.
			\label{item.sameorbit} 
		\end{enumerate}
	\end{lem}
	\begin{proof}
		We first prove \eqref{item.samefirstreturnmap} and \eqref{item.returntime}. Clearly $U_D(x)=T_D(x)=x$ for every $x\notin D$. Since $D$ is a fundamental domain for $P$, for all $x\in D$ and $i\in\{1,\dots,n(x)-1\}$, we have $P^i(x)\notin D$. Since $T_D(x)=x$ for all $x\notin D$, we deduce by induction that 
		\[U^i(x)=P^i(x)\notin D\text{ for all }x\in D\text{ and }i\in\{1,\dots,n(x)-1\}.\] 
		So for all $x\in D$, one has $U^{n(x)}(x)=UU^{n(x)-1}(x)=T_DP^{n(x)}(x)=T_D(x)$. This shows Items \eqref{item.samefirstreturnmap} and \eqref{item.returntime}.

		We now prove Item \eqref{item.sameorbit}. Clearly, $U\in [T]$. 
		We need to show that $T\in[U]$. Observe that for almost every $x$, the $U$-orbit of $x$ meets $D$: indeed, if $x\in P^i(D)$ for $i\in \{1,\dots,n(x)-1\}$, then $U^{-i}(x)=P^{-i}(x)\in D$. Since being in the same orbit is an equivalence relation, it is enough to show that any two points in $D$, which belong to the same $T$-orbit, are in the same $U$-orbit. This follows directly from \eqref{item.samefirstreturnmap}.
	\end{proof}
	
	We will also need the following lemma which can be proven with the same kind of arguments as above. 
	
	
	\begin{lem}\label{lem.P is Periodic}
		Let $U\in\Aut(X,\mu)$ and let $A$ be a measurable subset of $X$ which intersects every $U$-orbit. Then $(U_A)^{-1}U$ is periodic and $A$ is a fundamental domain for it.
	\end{lem}
	\begin{proof}
		Since $A$ intersects every $U$-orbit, for almost every $x\in X\setminus A$ there exists a smallest $n\geq 1$ such that $U^{-n}(x)\in A$. Remark that $((U_A)^{-1}U)^{-n}(x)=U^{-n}(x)\in A$ and hence $A$ intersects every $(U_A)^{-1}U$-orbit. If $x\in A$, then for every $0\leq n<n_{U,A}(x)$ we have that \[((U_A)^{-1}U)^n(x)=U^n(x)\notin A\text{ and }
		((U_A)^{-1}U)^{n_{U,A}(x)}(x)=U_A^{-1}UU^{n_{U,A}(x)-1}(x)=x.\]
		Since $A$ intersects every $(U_A)^{-1}U$-orbit, we obtain both that every $(U_A)^{-1}U$-orbit is finite and that $A$ is a fundamental domain for $(U_A)^{-1}U$.
	\end{proof}
	
	\subsection{\texorpdfstring{$\varphi$}{phi}-integrable orbit equivalence and full groups}
	
	We first define the notion of $\ph$-\textit{integrable orbit equivalence}.
	
	\begin{df}
		Fix $\varphi\colon \R_+\to\R_+$.
		Two aperiodic transformations $T_1,T_2\in\Aut(X,\mu)$ are $\ph$-\defin{integrably orbit equivalent} if there exists $S\in\Aut(X,\mu)$ such that $ST_1S^{-1}$ and $T_2$ have the same orbits and their respective cocycles are $\varphi$-integrable. To be more precise, we ask that
		\[\int_X \ph(\lvert c_{ST_1S^{-1}}(x)\rvert) d\mu<+\infty\text{ and }\int_X \ph(\lvert c_{T_2}(x)\rvert) d\mu<+\infty,\]
		where $c_{ST_1S^{-1}}$ is the $T_2$-cocycle of $ST_1S^{-1}$ and $c_{T_2}$ is the $ST_1S^{-1}$-cocycle of $T_2$, defined for all $x\in X$ by the equations
		\[ST_1S\inv(x) = T_2^{c_{ST_1S\inv}(x)}(x)\text{ and } T_2(x)=(ST_1S\inv)^{c_{T_2}(x)}(x).\]
	\end{df}
	
	When $\varphi(t)=t^p$ for some $p\in (0,+\infty)$, we recover the notion of $\LL^p$ orbit equivalence. 
	
	\begin{rmq}\label{rmk.not eq rel}
		We warn the reader that even though the term $\LL^p$ orbit equivalence is often used in the literature, this terminology may sound a bit deceptive. Indeed, since the integrability condition has no reason to be preserved under composition of orbit equivalences, we do not expect $\ph$-integrable (even $\LL^p$) orbit equivalence to be an equivalence relation for every concave function $\ph$, although we don't have any counterexample. 
		The fact that it is the case for $p=1$ seems to be a rather artificial consequence of Belinskaya's theorem.
	\end{rmq}
	
	In our work, the function $\ph$ is at most linear and for our main theorems the function is assumed to be \defin{sublinear}, that is $\lim_{t\to +\infty}\ph(t)/t=0$. For example we are interested in the case of $\LL^p$ orbit equivalence for $p\leq 1$, or in the case of $\ph(t)=\log(1+t)$.
	
	In the context of $\varphi$-orbit equivalence, it is natural to consider the the set of measure-preserving transformations $U$ whose cocycle $c_U$ is $\ph$-integrable. In order for this set to be a group, the following conditions on $\ph$ are required. 
	
	\begin{df}
		A function $\varphi\colon \R_+\to \R_+$ is \defin{metric-compatible} if: 
		\begin{itemize}
			\item (subadditivity) for all $s,t\in\R_+$, $\varphi(s+t)\leq \varphi(s)+\varphi(t)$.
			\item (separation) $\varphi(0)=0$ and $\varphi(t)>0$ for all $t>0$.
			\item (monotonicity) $\varphi$ is a non-decreasing function.
		\end{itemize}
	\end{df}
	
	\begin{example} 
		Any concave function $\varphi \colon \R_+\to\R_+$ that satisfies $\varphi(0)=0$ and $\varphi(t)>0$ for all $t>0$ is metric-compatible. In particular for every $p\leq 1$, the function $\varphi(t)=t^p$ is metric-compatible. It is moreover sublinear whenever $p<1$. Other examples of sublinear metric-compatible functions are given by $\ph(t)=\log(1+t)$ or $\ph(t)=t/\log(2+t)$.
	\end{example}
	
	The term ``metric-compatible'' was coined because of the following property: whenever $d$ is a metric on a set $X$, then $\ph\circ d$ is also a metric on $X$. 
	
	\begin{convention}
	Let $\varphi\colon \R_+\to\R_+$ be a metric-compatible function. For all $t\in\R$, we use the notation 
		\[\lvert t\rvert_\varphi\coloneqq\varphi(\lvert t\rvert).\]
		The map $(s,t)\mapsto\lvert s-t\rvert_\varphi$ is a metric on $\R$. 
	\end{convention}
	
	\begin{df}
		Let $\varphi\colon \R_+\to\R_+$ be a metric-compatible function. The \defin{$\varphi$-integrable full group} of an aperiodic transformation $T\in\Aut(X,\mu)$ is
		\[[T]_\ph\coloneqq\left\lbrace U\in [T]\colon 
		\int_X \lvert c_U(x)\rvert_\varphi d\mu<+\infty\right\rbrace,
		\]
		where $c_U \colon  X\to\Z$ denotes the $T$-cocycle of $U$. 
	\end{df}
	
	Given a metric-compatible function $\ph$, the $\ph$-integrable full group $[T]_{\varphi}$ is indeed a group: given $U,V\in [T]_{\varphi}$, the cocycle identity implies that \[c_{UV^{-1}}(x)=c_U(V^{-1}(x))+c_{V^{-1}}(x)=c_U(V^{-1}(x))-c_V(V^{-1}(x)).\] 
	We then get that 
	\begin{align}\label{eq.triangleineq}
		\int_X \lvert c_{UV^{-1}}(x)\rvert_\varphi d\mu&\leq \int_X \lvert c_U(V^{-1}(x))\rvert_\varphi d\mu+\int_X \lvert c_V(V^{-1}(x))\rvert_\varphi d\mu\nonumber \\
		&=\int_X \lvert c_U(x)\rvert_\varphi d\mu+\int_X \lvert c_V(x)\rvert_\varphi d\mu<+\infty.
	\end{align}
	
	\begin{example} \label{exmp.basic examples}
		If $\varphi$ is any metric-compatible function which is bounded, then $[T]_\varphi=[T]$ and if $\varphi$ is the identity map, then we recover the $\LL^1$ full group $[T]_1$ defined by the third named author in \cite{lemaitremeasurableanaloguesmall2018}. 
		Any other such $\ph$ gives rise to new\footnote{We can actually characterize when $[T]_\varphi=[T]_\psi$ and more generally when $[T]_\varphi\leq [T]_\psi$, see Proposition~\ref{prop: chara equal phifull}.} examples of full groups, such as  $\LL^p$ full groups $[T]_p$ for $0<p<1$ obtained with the function $\ph(t)=t^p$, or else $[T]_{\log}$ obtained with the function $\ph(t)=\log(1+t)$. 
	\end{example}
	
	\begin{rmq}
		Given a metric-compatible function $\varphi$, it is now straightforward to check that two aperiodic transformations $T_1,T_2\in\Aut(X,\mu)$ are $\ph$-integrably orbit equivalent if and only if there is $S\in\Aut(X,\mu)$ such that $ST_1S^{-1}\in[T_2]_\ph$ and $T_2\in[ST_1S^{-1}]_\ph$. However, the notion of $\varphi$-orbit equivalence is a priori weaker than conjugacy of $\varphi$-integrable full groups. 
		Indeed conjugacy of $\varphi$-integrable full groups is an equivalence relation but $\ph$-integrable orbit equivalence may not be, see Remark \ref{rmk.not eq rel}.
		This is in contrast with the case of classical orbit equivalence, see \cite[Thm.~4.1]{kechrisGlobalaspectsergodic2010}.
	\end{rmq}


	In our two main results, namely Theorem \ref{thmIntro:belinskaya opti asym} and \ref{thmIntro:belinskaya opti sym}, the sublinear function $\ph$ is not assumed to be metric-compatible. The following lemma will allow us to reduce to the case where $\ph$ is in addition metric-compatible.


	\begin{lem}\label{lem.free metric-compatible}
		Let $\ph\colon\R_+\to\R_+$ be a sublinear function. Then there is a sublinear metric-compatible function $\psi\colon\R_+\to\R_+$ such that $\varphi(t)\leq\psi(t)$ for all $t$ large enough. 
	\end{lem}
	\begin{proof}
		Set
		\begin{align*}
			\theta\colon\R_+^*\to\R_+,\quad & \theta(t)\coloneqq \min\left(1,\sup_{s\geq t}\frac{\ph(s)+1}{s}\right);\\
			\psi\colon\R_+\to\R_+,\quad & \psi(t)\coloneqq\int_0^t\theta(s)ds.
		\end{align*}
		Noting that $\theta$ is positive-valued and non-increasing, it is straightforward to check that $\psi$ is non-decreasing, subadditive and that $\psi(t)=0$ if and only if $t=0$. Moreover the fact that $\theta(t)$ tends to $0$ as $t$ approaches $+\infty$ implies that $\psi$ is sublinear.
		Now remark that for every $t\in\R_+^*$
		\[\psi(t)=\int_0^t \theta(s)ds\geq \int_0^t\theta(t)ds=t\theta(t).\]
		For $t\in \R_+^*$ large enough so that $\displaystyle\sup_{s\geq t}\frac{\ph(s)+1}s\leq 1$ we finally have \[t\theta(t)=t\sup_{s\geq  t}\frac{\ph(s)+1}{s}=t\sup_{s\geq  t}\frac{\ph(s)+1}{s}\geq t\frac{\ph(t)+1}{t}=\ph(t)+1\geq \varphi(t)\] so we are done.
	\end{proof}

	\begin{rmq}\label{rmk.restriction to metric-compatible}
		Given a sublinear function $\ph$, Lemma \ref{lem.free metric-compatible} grants us a sublinear metric-compatible function $\psi$ such that $\ph(t)\leq \psi(t)$ for all $t$ large enough. 
		Therefore, for any measurable function $f\colon X\to \Z$ we have 

        \[\int_X \psi(\lvert f(x)\rvert)d\mu<+\infty\text{ implies that  }\int_X\ph(\lvert f(x)\rvert)d\mu<+\infty.\]
		
		In particular, $\psi$-integrable orbit equivalence implies $\ph$-orbit equivalence and any element in a $\psi$-integrable full group will have $\ph$-integrable cocycle. 
	\end{rmq}
	
	We will state most of our results in the comfortable context of (sublinear) metric-compatible functions. However, many of our statements could be easily generalized to the general context of sublinear functions through Remark \ref{rmk.restriction to metric-compatible}. We will explicitly do so only in our main theorems, Theorem \ref{thmIntro:belinskaya opti asym} and \ref{thmIntro:belinskaya opti sym}.
	
	\subsection{Metric properties of \texorpdfstring{$\varphi$}{phi}-integrable full groups}
	
	We now introduce and study a natural extended pseudo-metric on full groups from which $\varphi$-integrable full groups naturally arise.
	
	\begin{lem} \label{lem.distanceonphifullgroup}
		Let $\varphi \colon  \R_+\to\R_+$ be a metric-compatible function and let $T\in\Aut(X,\mu)$ be an aperiodic transformation. Let $\d_{\varphi,T} \colon  [T]\times [T] \to \R_+\cup\{+\infty\}$ be the function defined by
		\[\d_{\varphi,T}(U,V)\coloneqq\int_X\lvert c_{U}(x)-c_V(x)\rvert_\varphi d\mu.\]
		Then the following are true.
		
		\begin{enumerate}[(i)]
			\item\label{item.fullgroupwithmetric} The group $[T]_\varphi$ is determined by $\d_{\ph,T}$: \[[T]_\ph=\{U\in [T]\colon \d_{\varphi,T}(U,\id)<+\infty\}.\]
			\item\label{item.truedistance} The restriction of $\d_{\varphi,T}$ to $[T]_\varphi\times[T]_\varphi$ is a metric on $[T]_\varphi$ which is right-invariant, that is, for all $U,V,W\in [T]_\varphi$, 
			\[\d_{\varphi,T}(UW,VW)=\d_{\varphi,T}(U,V).\]
		\end{enumerate}
	\end{lem}
	
	\begin{proof}
		Item \eqref{item.fullgroupwithmetric} is an immediate consequence of the definition of $[T]_\varphi$.
		
		Let us now prove Item \eqref{item.truedistance}. The fact that $\d_{\varphi,T}$ is a metric is a straightforward consequence of the fact that $(s,t)\mapsto \abs{s-t}_{\varphi}$ is a metric on $\R$. The right-invariance follows from the cocycle identity \eqref{eq:cocycle identity} and the fact that  the transformation $W$ is measure-preserving:
		\begin{align*}
			\d_{\varphi,T}(UW,VW)&=\int_X \abs{c_{UW}(x)-c_{VW}(x)}_\varphi d\mu\\
			&=\int_X \abs{c_U(W(x))-c_V(W(x))}_\varphi d\mu\\
			&=\int_X \abs{c_U(x)-c_V(x)}d\mu\\
			&=\d_{\varphi,T}(U,V).\qedhere
		\end{align*}
	\end{proof}
	
	\begin{example}\label{ex:basic examples d_phi} 
		Consider the metric-compatible function $\ph\coloneqq \min(\id_{\R_+},1)$. Then it is straightforward to check that $\d_{\ph,T}=d_u$ is the uniform metric on $[T]=[T]_\varphi$.
		
		Another example is obtained by taking $\varphi\coloneqq\id_{\R_+}$; we then recover the $\LL^1$ metric on the $\LL^1$ full group $[T]_1=[T]_{\varphi}$.
	\end{example}
	
	In order to compare $\ph$-integrable full groups, we are led to compare asymptotically metric-compatible functions. 
	We will use the following standard notation: given two real-valued functions $f$ and $g$, we write $f(t)=O(g(t))$ as $t\to +\infty$ if there exist $t_0>0$ and $C>0$ such for all $t>t_0$, we have  $\lvert f(t)\rvert\leq C \lvert g(t)\rvert$.
	Since the functions we consider are subadditive, it is enough to compare them on the integers.
	
	\begin{lem}\label{lem.lipschitzonN}
		Let $\ph$, $\psi$ be two metric compatible functions. Then the following are equivalent.
		\begin{enumerate}[(i)]
			\item\label{item.phi inferieur a psi} $\ph(t)=O(\psi(t))$ as $t\to +\infty$.
			\item\label{item.phi et psi geq 1} There exists $C>0$ such that $\ph(t)\leq C\psi(t)$ for all $t\geq 1$. 
			\item\label{item.phi et psi sur N} There exists $C>0$ such that $\ph(k)\leq C\psi(k)$ for all integers $k\in\N$.
		\end{enumerate}
	\end{lem}
	
	\begin{proof}
		We first prove that \eqref{item.phi inferieur a psi} implies \eqref{item.phi et psi geq 1}. 
		Let $t_0\geq 1$ and $D>0$ such that for all  $t>t_0$, we have $\ph(t)\leq D\psi(t)$. Set $C\coloneqq \max(D,\ph(t_0)/\psi(1))$ and observe that since $\varphi$ and $\psi$ are non-decreasing, $\ph(t)\leq C\psi(t)$ for all $t\geq 1$.
		
		The implication \eqref{item.phi et psi geq 1}$\implies$\eqref{item.phi et psi sur N} is straightforward, so we are left with proving
		\eqref{item.phi et psi sur N}$\implies$\eqref{item.phi inferieur a psi}. Let $C>0$ be such that $\ph(k)\leq C\psi(k)$ for all integers $k\in\N$. Fix a real number $t\geq 2$ and let $n\in\N^*$ be such that $n\leq t<n+1$. Then we have
		\[\ph(t)\leq \ph(n+1)\leq C\psi(n+1)\leq C(\psi(t)+\psi(1))\leq C\left(1+\frac{\psi(1)}{\psi(t)}\right)\psi(t),\]
		and since $\psi(t)\geq \psi(1)$ for every $t\geq 1$, the proof is complete.
	\end{proof}
	
	We now compare $\varphi$-integrable full groups for different metric-compatible functions.
	
	\begin{lem}\label{lemma.comparisontopo}
		Let $\ph$ and $\psi$ be two metric-compatible functions and fix an aperiodic transformation $T\in\Aut(X,\mu).$
		If $\ph(t)=O(\psi(t))$ as $t\to+\infty$, then $[T]_\psi\leq [T]_\ph$. Moreover, the inclusion map is Lipschitz.
	\end{lem}
	
	\begin{proof}
		By Lemma \ref{lem.lipschitzonN}, there is $C>0$ such that $\ph(k)\leq C\psi(k)$ for all integers $k\in\N$. Let $U,V\in [T]$. Then for almost every $x\in X$, 
		\[\lvert c_U(x)-c_V(x)\rvert_\ph\leq C \lvert c_U(x)-c_V(x)\rvert_\psi.\]
		Integrating over $X$, we get that $\d_{\ph,T}(U,V)\leq C \d_{\psi,T}(U,V)$. The lemma now follows immediately.
	\end{proof}
	
	\begin{cor}\label{cor: inclusions}
		Whenever $T$ is an aperiodic measure-preserving transformation, we have \[
		[T]_1\leq [T]_\varphi\leq [T],\]
		and the inclusion maps are Lipschitz,
	\end{cor}
	\begin{proof}
		Since $\ph$ is subadditive, we have $\ph(t)=O(t)$ as $t\to +\infty$. Moreover $\min(1,t)=O(\ph(t))$ as $t\to +\infty$. The conclusion now follows from Lemma \ref{lemma.comparisontopo}.
	\end{proof}
	
	We will show in Proposition \ref{prop: chara equal phifull} that the implication in Lemma \ref{lemma.comparisontopo} is an equivalence. For this, we will make a crucial use of the fact that the topologies induced by these metrics are Polish group topologies, see Theorem \ref{thm.polishgrouptopology}.

	\begin{rmq}
		Let $d_T \colon  X\times X\to \R_+\cup\{+\infty\}$ be the extended metric on $X$ defined by
		\[d_T(x,y)\coloneqq \inf\{n\in\N\colon T^n(x)=y\text{ or }T^n(y)=x\}.\]
		Then by definition of the $T$-cocycle of any $U\in [T]$, we have that for all $x\in X$, $d_T(U(x),x)=\lvert c_U(x)\rvert$. For all $U,V\in [T]_\varphi$, the cocycle identity implies that $c_{UV^{-1}}(x)=c_U(V^{-1}(x))-c_V(V^{-1}(x))$. Since $V$ preserves the measure, we obtain
		\begin{align*}
			\d_{\varphi,T}(U,V)&=\int_X\lvert c_U(V^{-1}(x))-c_V(V^{-1}(x))\rvert_\ph d\mu \\
			&=\int_X \lvert c_{UV^{-1}}(x)\rvert_\ph d\mu \\
			&=\int_X \ph(d_T(UV^{-1}(x),x))d\mu\\
			&=\int_X\ph(d_T(U(x),V(x)))d\mu.
        \end{align*}
		We won't use this formula hereafter. However, this point of view allows one to define $\varphi$-integrable full groups of non-necessarily free actions of finitely generated groups.
		Some of the arguments given in this paper work in this wider context; this will be examined in an upcoming work.
	\end{rmq}

	\section{Flexibility of \texorpdfstring{$\varphi$}{phi}-integrable orbit equivalence}\label{sec.belinskaya}

	\subsection{Construction of cycles in \texorpdfstring{$\varphi$}{phi}-integrable full groups}\label{sub.cycle}
	
	An \defin{$n$-cycle}, $n\geq 2$, is a periodic transformation $P\in\Aut(X,\mu)$ whose orbits have cardinality either $1$ or $n$. The aim of this section is to prove the following result. 
	
	\begin{thm}\label{thm.existencecycle}
		Let $\varphi\colon \R_+\to\R_+$ be a sublinear metric-compatible function. Let $T\in\Aut(X,\mu)$ be aperiodic. Then for all measurable $A\subseteq X$ and all integers $n\geq 2$, there exists an $n$-cycle $P\in [T]_\varphi$ whose support is equal to $A$. 
	\end{thm}	
	
	\begin{rmq}
		The hypothesis that $\ph$ is sublinear is necessary, as the result is false  for $\mathrm{L}^1$ full groups of certain aperiodic transformations. Indeed if $T\in\Aut(X,\mu)$ is ergodic, then there exists an $n$-cycle in $[T]_1$ whose support is $A\subseteq X$ if and only if $\exp(2i\pi/n)$ is in the spectrum of the restriction of $T_A$ to $A$ \cite[Thm.~4.8]{lemaitremeasurableanaloguesmall2018}. In particular the $\mathrm{L}^1$ full group of the Bernoulli shift contains no $n$-cycle with full support for any $n\geq 2$. By contrast, Theorem \ref{thm.existencecycle} says that as soon as $p<1$, its $\mathrm{L}^p$ full group contains an $n$-cycle of full support for every $n\geq 2$.
	\end{rmq}
	
	\begin{example}\label{ex: bernoulli involutions}
		In certain concrete situations, we can exhibit explicit involutions. Let $T$ be the Bernoulli shift on $(\{0,1\},\kappa)^{\otimes\Z}$, where $\kappa$ is the uniform measure on $\{0,1\}$. Then for every $0<p<1/2$, there exists an involution in $[T]_p$ with full support and fundamental domain $X_0\coloneqq \{(x_n)_{n\in\Z}\in \{0,1\}^\Z\colon x_0=0\}$. 
		
		Indeed, for all $x\in X_0$, let $N(x)$ be the infimum of $n\geq 1$ such that $1$ appears strictly more often than $0$ in $\{x_1,\dots,x_n\}$. 
		Then the map $\pi\colon x\in X_0\mapsto T^{N(x)}(x)\in \{0,1\}^\mathbb Z\setminus X_0$ is almost everywhere well-defined and injective. 
		Thus it can be extended to an involution $P\in [T]$ with full support and fundamental domain $X_0$. 
		Standard estimates on the first return time to $0$ of the simple random walk on $\Z$ \cite[Chap.~III.2]{Feller} imply that $P$ belongs to $[T]_p$ for all $0<p<1/2$. 
	\end{example}
	
	\begin{rmq}Theorem \ref{thm.existencecycle} tells us that any measurable subset $A\subseteq X$ is the support of an involution. The situation is less flexible regarding fundamental domains. For example, the subset $X_0$ introduced in the previous example cannot be the fundamental domain of any involution in the $\LL^p$ full group of the Bernoulli shift for $1/2\leq p\leq 1$, as a consequence of a result of Liggett
		\cite{liggettTaggedParticleDistributions2002}.
		Note that his result is more general and stated in probabilistic terms; the connection to our context and a purely ergodic-theoretic version of his proof is presented
		in the second named author's PhD thesis \cite[Chap.~3]{josephTopologicalMeasurableDynamics2022}.
	\end{rmq}
	
	A \defin{partial measure-preserving transformation} of $(X,\mu)$ is a bimeasurable measure-preserving bijection $\pi$ between two measurable subsets $\dom(\pi)$ and $\rng(\pi)$ of $X$, called respectively the \textit{domain} and the \textit{range} of $\pi$. The \defin{support} of $\pi$ is the set 
	\[\supp(\pi)\coloneqq\{x\in\dom(\pi)\colon \pi(x)\neq x\}\cup\{x\in\rng(\pi)\colon \pi^{-1}(x)\neq x\}.\]
	A \defin{pre-cycle} of length $n\geq 2$ is a partial measure-preserving transformation $\pi\colon  \dom(\pi)\to\rng(\pi)$ of $(X,\mu)$ such that if we set $B\coloneqq \dom(\pi)\setminus\rng(\pi)$, then 
	\begin{itemize}
		\item $\{\pi^0(B),\ldots,\pi^{n-2}(B)\}$ is a partition of $\dom(\pi)$,
		\item $\{\pi^1(B),\ldots,\pi^{n-1}(B)\}$ is a partition of $\rng(\pi)$.
	\end{itemize}
	The set \(B=\dom(\pi)\setminus\rng(\pi)\) is called the \defin{basis} of the pre-cycle $\pi$.

	A pre-cycle $\pi$ of length $n$ can be extended to an $n$-cycle $P$ and called the \defin{closing cycle} of $\pi$, as follows: 
	\[P(x)\coloneqq\left\{\begin{array}{ll}\pi(x) & \text{ if }x\in\dom(\pi),\\ \pi^{-(n-1)}(x)& \text{ if }x\in \rng(\pi)\setminus\dom(\pi), \\ x & \text{ else.}\end{array}\right.\]
	Observe that the support of $P$ coincides with the support of the pre-cycle $\pi$ and that the basis $B$ is a fundamental domain for the restriction of $P$ to its support.  A pre-cycle $\pi$ is \defin{induced} by $T\in\Aut(X,\mu)$ if for all $x\in\dom(\pi)$, we have $\pi(x)=T_{\supp(\pi)}(x)$.
	
	\begin{lem}\label{lem.cycleinduced} 
		Let $\varphi \colon  \R_+\to\R_+$ be a metric-compatible function. Let $T\in\Aut(X,\mu)$ be an aperiodic transformation, let $\pi$ be a pre-cycle induced by $T$ and let $P$ be its closing cycle. Then 
		\[\d_{\varphi,T}(P,\id)\leq 2 \d_{\varphi,T}(T_{\supp(\pi)},\id).\]
		In particular $P$ belongs to $[T]_\varphi$.
	\end{lem}
	\begin{proof} Let $n$ be the length of the pre-cycle $\pi$, let $A\coloneqq\supp(\pi)$ and let $B\coloneqq \dom (\pi)\setminus\rng(\pi)$ the basis of $\pi$. Since $\pi$ is induced by $T$, for all $x\in \dom(\pi)$, one has $\pi(x)=P(x)=T_A(x)$. This implies that $c_P(x)=c_{T_A}(x)$ for all $x\in\dom(\pi)$. Thus,
		\begin{align*}
			\d_{\varphi,T}(P,\id) &= \int_{\dom(\pi)}\lvert c_{T_A}(x)\rvert_\varphi d\mu+ \int_{P^{n-1}(B)}\lvert c_{P^{-(n-1)}}(x)\rvert_\varphi d\mu\\
			&\leq \d_{\varphi,T}(T_{A},\id) + \int_{B}\lvert c_{P^{n-1}}(x)\rvert_\varphi d\mu.
		\end{align*}
		Moreover, for all $x\in B$, the cocycle identity yields
		\begin{align*}
			\lvert c_{P^{n-1}}(x)\rvert_\varphi 
			&\leq \lvert c_P(x)\rvert_\ph + \lvert c_P(P(x))\rvert_\ph +\dots \lvert c_P(P^{n-2}(x))\rvert_\ph.
		\end{align*}
		We now use the fact that $P$ preserves the measure and that $\dom(\pi)=B\sqcup P(B)\sqcup\dots\sqcup P^{n-2}(B)$ to get
		\[\int_{B}\lvert c_{P^{n-1}}(x)\rvert_\varphi d\mu\leq \int_{\dom(\pi)}\lvert c_{P}(x)\rvert_\varphi d\mu \leq \int_X\lvert c_{T_A}(x)\rvert_\ph d\mu,\]
		which concludes the proof. 
	\end{proof}
	
	Kac's lemma, that is Equation \eqref{equation.kac}, implies that for every measurable $A\subseteq X$, the first return map $T_A$ belongs to $[T]_1$, which is contained in $[T]_\ph$ for every metric-compatible function $\ph$ by Corollary \ref{cor: inclusions}. We will need a more quantitative version of this fact.
	
	\begin{lem}\label{lem.distanceinduiteaid}
		Let $\varphi \colon  \R_+\to\R_+$ be a metric-compatible function. Let $T\in\Aut(X,\mu)$ be an aperiodic transformation, let $A\subseteq X$ be a measurable subset and let $C>0$. Then 
		$$\mathsf \d_{\varphi,T}(T_A,\id)\leq C\varphi(1) \mu(A) + \sup_{t> C}\frac{\varphi(t)}{t}.$$
	\end{lem}

	\begin{proof}
		Recall that the $T$-cocycle of $T_A$ is the return time $n_{T,A}$, which is non-negative. Set $B\coloneqq\{x\in A\colon n_{T,A}(x)\leq C\}$. We have
		\begin{align*}
			\d_{\varphi,T}(T_A,\id) 
			&=\int_{B}\varphi(n_{T,A}(x))d\mu+\int_{A\setminus B}\varphi(n_{T,A}(x))d\mu \\
			&\leq C\varphi(1)\mu(B)+\int_{A\setminus B}\frac{\varphi(n_{T,A}(x))}{n_{T,A}}n_{T,A}(x)d\mu\\
			&\leq C\varphi(1)\mu(A)+\left(\sup_{t> C}\frac{\varphi(t)}{t}\right)\int_{A\setminus B}n_{T,A}(x)d\mu
		\end{align*}
		and the last integral is at most $1$ by Kac's lemma, see Equation \eqref{equation.kac}.
	\end{proof}
	
	\begin{cor}\label{cor.distanceinduiteaid}
		Let $\varphi \colon  \R_+\to\R_+$ be a sublinear metric-compatible function and let $T\in\Aut(X,\mu)$ be an aperiodic transformation.
		Then $\d_{\ph, T}(T_A,\id)$ tends to $0$ as $\mu(A)$ approaches $0$.
	\end{cor}
	
	\begin{proof}
		Fix $\varepsilon >0$. By sublinearity, let $C>0$ such that for all $t>C$, we have $\ph(t)/t<\varepsilon$. For all measurable $A\subseteq X$, if $\mu(A)<\varepsilon/C\ph(1)$, then $\d_{\ph,T}(T_A,\id)<2\varepsilon$, which concludes the proof. 
	\end{proof}
	
	\begin{rmq}
		In particular, by taking $\varphi$ bounded, we recover the well-known fact that $d_u(T_A,\id)$ tends to $0$ as $\mu(A)$ approaches $0$ (see Lemma \ref{lem: keane continuity}).
	\end{rmq}

	The following lemma is a direct consequence of Rokhlin's lemma.
	
	\begin{lem}\label{lem.existencepre-cycle}
		Let $T\in\Aut(X,\mu)$ be aperiodic and $A\subseteq X$ be measurable. For all $\varepsilon >0$ and all integers $n\geq 2$, there exists a pre-cycle $\pi$ of length $n$, induced by $T$, such that $\supp(\pi)\subseteq A$ and $\mu(A\setminus\supp(\pi))\leq\varepsilon$. 
	\end{lem}
	\begin{proof}
		Since $T$ is aperiodic, $T_A$ is aperiodic on its support.  
		We apply Rokhlin's lemma to $T_A$ to find a measurable subset $B\subseteq A$ such that $B,T_A(B),\dots, (T_A)^{n-1}(B)$ are pairwise disjoint and 
		\[\mu\Big(A\setminus \big(B\sqcup\dots\sqcup (T_A)^{n-1}(B)\big)\Big)\leq\varepsilon.\]
		Then the restriction of $T_A$ to $B\sqcup\dots\sqcup (T_A)^{n-2}(B)$ is a pre-cycle of length $n$, which is induced by $T_A$ and thus by $T$. Finally, its support satisfies the desired assumptions.
	\end{proof}
	
	We are now ready to prove the existence of $n$-cycles with prescribed support in $\varphi$-integrable full groups.

	\begin{proof}[Proof of Theorem \ref{thm.existencecycle}]
		Let $T\in\Aut(X,\mu)$ be an aperiodic element, let $A\subseteq X$ be a measurable subset and let $n\geq 2$. Since $\ph$ is sublinear, we can and do fix a sequence $(C_k)_{k\geq 1}$ of strictly positive numbers such that 
		\[\sup_{t>C_k}\frac{\varphi(t)}{t}\leq 2^{-k}\text{ for all }k\geq 1.\]
		Then, we use Lemma \ref{lem.existencepre-cycle} to construct inductively a sequence $(\pi_k)_{k\geq 0}$ of pre-cycles of length $n$ induced by $T$, whose supports are pairwise disjoint subsets of $A$ and such that for all $k\geq 1$,
		\[\mu\Big(A\setminus \big( \supp(\pi_0)\sqcup\dots\sqcup\supp(\pi_{k-1})\big)\Big)\leq \frac{1}{2^{k}C_{k}}.\]
		This inequality implies in particular that for all $k\geq 1$, we have $\mu(\supp(\pi_k))\leq 1/(2^kC_k)$. Let $P_k$ be the closing cycle of $\pi_k$ and let $P\in\Aut(X,\mu)$ be the $n$-cycle defined by $P(x)\coloneqq P_k(x)$ for $x\in\supp(P_k)$ and $P(x)\coloneqq x$ for $x\notin A$. The support of $P$ is equal to $A$ and by Lemma \ref{lem.cycleinduced} and \ref{lem.distanceinduiteaid}, we have
		\begin{align*}
			\d_{\varphi,T}(P,\id)& =\sum_{k\geq 0}\d_{\varphi,T}(P_k,\id) \\
			&\leq 2 \sum_{k\geq 0}\d_{\varphi,T}(T_{\supp(\pi_k)},\id) \\
			&\leq 2\d_{\varphi,T}(T_{\supp(\pi_0)},\id)+2\sum_{k\geq 1}\left(\varphi(1)C_k\mu(\supp(\pi_k))+\sup_{t>C_k}\frac{\varphi(t)}{t}\right).
		\end{align*}
		The second term is by construction a converging series, so we are done. 
	\end{proof}

	\subsection{Construction of \texorpdfstring{$\varphi$}{phi}-integrable orbit equivalences}\label{sub.phifullgroup}
	
	Let us now prove Theorem \ref{thmIntro:belinskaya opti sym}.
	
	\begin{thm}\label{thm.phifullgroup} Let $\varphi \colon \R_+\to\R_+$ be a sublinear function. Let $T\in\Aut(X,\mu)$ be ergodic. For all $n\geq 2$, there exists $U\in\Aut(X,\mu)$ such that $T$ and $U$ are $\varphi$-integrably orbit equivalent and $U^n$ is not ergodic.
	\end{thm}
	\begin{proof}
		By Lemma \ref{lem.free metric-compatible}, there is a sublinear metric-compatible function $\psi$ such that $\ph(t)\leq \psi(t)$ for all $t$ large enough. 
		In particular, $\psi$-integrable orbit equivalence implies $\ph$-orbit equivalence (cf.\ Remark \ref{rmk.restriction to metric-compatible}). Hence if the theorem holds for $\psi$ then it holds for $\ph$.
		Therefore, by replacing $\ph$ by $\psi$, we may and do assume that $\ph$ is a metric-compatible function. 
		
		By Theorem \ref{thm.existencecycle}, there exists an $n$-cycle $P\in [T]_\varphi$ whose support is $X$. We fix a fundamental domain $D$ for $P$ and we let $U\coloneqq T_DP$. 
		By Lemma \ref{lem.sameorbits} the following hold:
		\begin{itemize}
			\item the first return maps $U_D$ and $T_D$ coincide: $U_D=T_D$;
			\item for all $x\in D$, we have \(U_D(x)=U^n(x)\);
			\item $T$ and $U$ have same orbits.
		\end{itemize}
		By the second item, the set $D$ is $U^n$-invariant. So $U^n$ is not ergodic.
		
		We will now prove that $T$ and $U$ are $\varphi$-integrably orbit equivalent. Since $T$ and $U$ have same orbits, we are left to show that $T\in [U]_\varphi$ and $U\in [T]_\varphi$. As a direct consequence of Kac's lemma, see Equation \eqref{equation.kac}, we have that $T_D\in [T]_1\leq [T]_\varphi$ and therefore $U=T_DP\in [T]_\varphi$. 
		
		We now prove that $T\in [U]_\varphi$. 
		In the sequel, if a measure-preserving transformation $V$ belongs to $[T]=[U]$, we shall denote by $c^T_V$ the $T$-cocycle of $V$ and by $c_V^U$ the $U$-cocycle of $V$. 
		\begin{claim*} Let $V\in [T]$. Then for all $y\in D$ such that $V(y)\in D$,
			\[\left\lvert c_V^U(y)\right\rvert \leq n\left\lvert c^T_V(y)\right\rvert.\]
		\end{claim*}
		\begin{proof}[Proof of Claim]
			Note that since $y$ and $V(y)$ belong to $D$, any $i\in \Z$ such that $U^i(z)=V(z)$ must be a multiple of $n$.
			If we combine this with the fact that $U_D(z)=U^n(z)$ for all $z\in D$ and that \mbox{$U_D=T_D$}, we obtain:
			\begin{align*}
				\left\lvert c^U_{V}(y)\right\rvert=&\min\lbrace |i| \colon U^i(y)=V(y)\rbrace\\
				=& n\min\lbrace |i|\colon U_D^i(y)=V(y)\rbrace\\
				=& n\min\lbrace |i|\colon T_D^i(y)=V(y)\rbrace\\
				\leq & n\min\lbrace |i|\colon T^i(y)=V(y)\rbrace\\
				\leq & n\left\lvert c^T_{V}(y)\right\rvert.\qedhere
			\end{align*}
		\end{proof}
		Let $x\in X$.
		By definition of $U$, there are two integers $0\leq k,l\leq n-1$ such that $U^k(x)\in D$ and $U^l(T(x))\in D$.
		By the cocycle identity,
		\[c^U_{U^lTU^{-k}}(U^k(x))=c^U_{U^l}(T(x))+c^U_T(x)+c^U_{U^{-k}}(U^{k}(x))=l+c^U_T(x)-k.\]
		Hence
		\[\left\lvert c^U_T(x) \right\rvert\leq \left\lvert c^U_{U^lTU^{-k}}(U^k(x))\right\rvert+n.\]
		Using the claim for $V=U^{l}TU^{-k}$ and $y=U^k(x)$, we obtain
		\[\left\lvert c^U_T(x)\right\rvert\leq n\left\lvert c^T_{U^lTU^{-k}}(U^k(x))\right\rvert+n.\]
		We now apply $\varphi$, use its subadditivity and integrate over $X$ to get 
		\[ \int_X\left\lvert c^U_{T}(x)\right\rvert_\ph d\mu\leq \max_{0\leq k,l\leq n-1}\int_{X} n\left\lvert c^T_{U^lTU^{-k}}(U^k(x))\right\rvert_\ph d\mu+\varphi(n),\]
		which is bounded since $U^lTU^{-k}\in [T]_\ph$. Hence $T\in[U]_\varphi$ and this concludes the proof of the theorem.
	\end{proof}
	
	The following direct corollary says that the analogue of Belinskaya's theorem for $\varphi$-integrable orbit equivalence does not hold as soon as $\varphi$ is sublinear.
	
	\begin{cor}\label{cor.belinskayaoptimal} Let $\varphi \colon \R_+\to\R_+$ be a sublinear function. Let $T\in\Aut(X,\mu)$ be an ergodic transformation and assume that $T^n$ is ergodic for some $n\geq 2$. Then there exists $U\in\Aut(X,\mu)$ such that $T$ and $U$ are $\varphi$-integrably orbit equivalent but not flip-conjugate.
	\end{cor}
	\begin{proof}
		Let $n\geq 2$ be such that $T^n$ is ergodic. 
		By the previous theorem, we find $U\in\Aut(X,\mu)$ such that $U$ is $\varphi$-integrably orbit equivalent to $T$ and $U^n$ is not ergodic. In particular $U$ cannot be flip-conjugate to $T$ because otherwise $U^n$ would be flip-conjugate to $T^n$, which is ergodic.
	\end{proof}
	
	\begin{rmq}
	     Note that $T$ and $U$ do not have the same spectrum, as the spectrum of $U$ contains $\exp(2i\pi/n)$ whereas the spectrum of $T$ doesn't. So the spectrum is not an invariant of $\varphi$-integrable orbit equivalence.
	\end{rmq}
	
	\begin{qu}\label{qu.neverrigid}
		Let $\varphi \colon  \R_+\to\R_+$ be a sublinear metric-compatible function. Let $T\in\Aut(X,\mu)$ be an ergodic transformation such that $T^n$ is non-ergodic for all $n\geq 2$. Does there exist $U\in\Aut(X,\mu)$ such that $T$ and $U$ are $\varphi$-integrably orbit equivalent but not flip-conjugate?
	\end{qu}
	
	As we will see in Section \ref{sub.aper}, a weaker result holds in full generality: for every ergodic $T\in\Aut(X,\mu)$ and every sublinear metric-compatible function, there is $U\in[T]_\varphi$ such that $U$ and $T$ have the same orbits, but are not flip-conjugate. This 
	relies on the Baire category theorem, using the fact that $[T]_\varphi$ is a Polish group (see Section \ref{sec.polish}).
	
	\subsection{Connection to Shannon orbit equivalence}
	
	Let $I$ be a countable set and $f\colon X\to I$ a measurable map. The Shannon entropy of $f$ is the quantity 
	\[H(f)\coloneqq-\sum_{i\in I}\mu(f\inv\{i\})\log\mu(f\inv\{i\}).\]
	
	\begin{df}[Kerr-Li]
		Two aperiodic transformations $T_1,T_2\in\Aut(X,\mu)$ are \defin{Shannon orbit equivalent} if there exists $S\in\Aut(X,\mu)$ such that $ST_1S\inv$ and $T_2$ have the same orbits and
		\[H(c_{ST_1S^{-1}})<+\infty\text{ and }H(c_{T_2})<+\infty,\]
		where $c_{ST_1S^{-1}}$ is the $T_2$-cocycle of $ST_1S^{-1}$ and $c_{T_2}$ is the $ST_1S^{-1}$-cocycle of $T_2$.
	\end{df}
	
	\begin{lem} There are two positive constants $C_1,C_2>0$ such that for any measurable function $f \colon  X\to\Z$, we have 
		\[H(f)\leq C_1\int_X \log (1+\lvert f(x)\rvert)d\mu + C_2.\]
	\end{lem}
	The proof we propose is inspired by a classical proof that integrable functions have finite Shannon entropy, see for instance 
	\cite[Lem.~2.1]{austinBehaviourEntropyBounded2016} or \cite[Fact~1.1.4]{downarowiczEntropyDynamicalSystems2011}.
	
	\begin{proof} Let $f_+\coloneqq\max(f,0)$ and $f_-\coloneqq\min(f,0)$, so that $f=f_++f_-$. We have \[\int_X\log(1+\lvert f\rvert)d\mu=\int_X\log(1+f_+)d\mu +\int_X\log(1-f_-)d\mu.\] 
		By subadditivity, see for instance \cite[Chap.~1]{downarowiczEntropyDynamicalSystems2011}, \[H(f)=H(f_++f_-)\leq H(f_+)+H(-f_-).\] Hence, it is enough to prove the lemma for $f \colon  X\to\N$. So let us fix $f:X\to\N$. For all $n\in\N$, let \mbox{$p_n\coloneqq\mu(f^{-1}\{n\})$}. By definition of the Shannon entropy,
		\[H(f)=-\sum_{n\geq 0}p_n\log p_n.\]
		By elementary calculus, one checks that for all $t>0$ and $s\in \R$, $-t\log t\leq st+e^{-s-1}$. Applying this for $t=p_n$ and $s=2\log(n+1)$ and summing over $n$, we get
		\[H(f)\leq 2\sum_{n\geq 0}p_n\log(1+n)+\sum_{n\geq 0}\frac{e^{-1}}{(n+1)^2}.\]
		To conclude, we observe that $\sum_{n\geq 0}p_n\log(1+n)=\int_X \log (1+ f(x))d\mu.$ 
	\end{proof}
	
	We immediately deduce the following comparison between $\varphi$-integrable orbit equivalence and Shannon orbit equivalence.
	
	\begin{thm}\label{thm.intOEimpliesShannonOE}
		Let $\varphi:\R_+\to\R_+$ be a function such that $\log(1+t)=O(\ph(t))$ as $t\to +\infty$. Then for any aperiodic transformation $T\in\Aut(X,\mu)$, every $S\in[T]$ whose  $T$-cocycle is $\varphi$-integrable has finite Shannon entropy. 
		
		In particular, if two aperiodic transformations $S,T\in\Aut(X,\mu)$ are $\varphi$-integrably orbit equivalent, then they are Shannon orbit equivalent.
	\end{thm}
	\begin{rmq}
		Note that for every $p\in(0,+\infty)$, we have $\log(1+t)=O(t^p)$ as $t\to +\infty$. Therefore $\LL^p$ orbit equivalence implies Shannon orbit equivalence for measure-preserving transformations.
	\end{rmq}
	In \cite{kerrEntropyShannonOrbit2019}, Kerr and Li asked whether Shannon orbit equivalence of ergodic transformations implies flip-conjugacy. We prove that it is not the case.
	
	\begin{thm}\label{thm.reponsekerrli}
		Let $T\in\Aut(X,\mu)$ be an ergodic transformation, assume that $T^n$ is ergodic for some $n\geq 2$. Then there exists $U\in\Aut(X,\mu)$ such that $T$ and $U$ are Shannon orbit equivalent but not flip-conjugate. 
	\end{thm}
	
	\begin{proof}
		Let us consider the sublinear metric-compatible function $\varphi\colon \R_+\to\R_+$ given by $\varphi(t)\coloneqq \log(1+t)$. By Corollary \ref{cor.belinskayaoptimal}, there exists $U\in\Aut(X,\mu)$ such that $T$ and $U$ are $\varphi$-integrably orbit equivalent, but not flip-conjugate. On the other hand, the transformations $T$ and $U$ are Shannon orbit equivalent by Theorem \ref{thm.intOEimpliesShannonOE}.
	\end{proof}
	
	\subsection{Finiteness of entropy and Shannon orbit equivalence}
	
	Kerr and Li implicitly asked whether dynamical entropy is an invariant of Shannon orbit equivalence for ergodic measure-preserving transformations. 
	Shortly after a first version of our paper appeared, they obtained a positive answer \cite[Thm.~A]{kerr2022entropy}.
	In this section, we provide a short proof that \emph{finiteness} of dynamical entropy is an invariant of Shannon orbit equivalence. 
	We start by recalling a definition of dynamical entropy of measure-preserving transformations which is convenient for our purposes.
	
	\begin{df}
		Let $T\in\Aut(X,\mu)$. A measurable map $f\colon X\to I$, where $I$ is countable, is called $T$\defin{-dynamically generating} if there is a full measure set $X_0\subseteq X$ such that for all distinct $x,y\in X_0$, there is $n\in\Z$ such that 
		\( f(T^n(x))\neq f(T^n(y)).\)
	\end{df}
	
	\begin{df}
		The \defin{dynamical entropy} of a measure-preserving transformation $T\in \Aut(X,\mu)$ is the infimum of the Shannon entropies of its $T$-dynamically generating functions.
	\end{df}
	
	The above definition is not the standard definition, however it is equivalent by a theorem of Rokhlin \cite[Thm.~10.8]{rokhlinLecturesentropytheory1967}. Also note that by definition, the dynamical entropy of $T\in\Aut(X,\mu)$ is finite if and only if $T$ admits a dynamically generating function of finite entropy.
	
	\begin{prop}\label{prop.finiteness entropy preserved}
		Let $T\in\Aut(X,\mu)$ be an aperiodic transformation with infinite dynamical entropy and let $U\in [T]$ be a transformation whose $T$-cocycle has finite Shannon entropy. Then $U$ has infinite dynamical entropy.
	\end{prop}
	\begin{proof}
		Let $f\colon X\to I$ be a $U$-dynamically generating function and denote by $c_U$ the $T$-cocycle of $U$. We claim that the couple $(f,c_U)\colon X\to I\times\Z$ is $T$-dynamically generating.
		Indeed, let $x,y\in X$ such that 
		\[c_U(T^n(x))=c_U(T^n(y))\text{ and }f(T^n(x))=f(T^n(y))\text{ for all }n\in\Z.\]
		The first equality and the cocycle identity imply that $c_{U^n}(x)=c_{U^n}(y)$ for all $n\in\Z$. So for all $n\in\Z$
		\[f(U^{n}(x))=f(T^{c_{U^n}(x)}(x))= f(T^{c_{U^n}(y)}(y))=f(U^{n}(y)).\]
		Since $f$ is $U$-dynamically generating, the above equation implies that $x=y$. 
		
		Assume by contradiction that $U$ has finite entropy.
		This implies that there exists a $U$-generating function $f$ with finite Shannon entropy. Since $c_U$ has finite Shannon entropy,  then so does the function $(f,c_U)$. But we have seen that $(f,c_U)$ is a $T$-generating function, hence we deduce that $T$ has finite dynamical entropy and the proof is complete.
	\end{proof}
	
	\begin{cor}
		Suppose $T_1$, $T_2$ are two aperiodic measure-preserving transformations which are Shannon orbit equivalent. Then $T_1$ has finite dynamical entropy if and only if $T_2$ has finite dynamical entropy.
	\end{cor}
	
	\begin{rmq}\label{rmk.finiteness of entropy is oe-invariant}
	    As explained before, Kerr and Li recently obtained that dynamical entropy itself is preserved under Shannon orbit equivalence \cite[Thm.~A]{kerr2022entropy}.
		Suppose now that $\varphi:\R_+\to\R_+$ is a function that satisfies \mbox{$\log(1+t)= O(\ph(t))$} as $t\to +\infty$, such as $\ph(t)=\log(1+t)$ or $\ph(t)=t^p$. By Kerr and Li's result and Theorem \ref{thm.intOEimpliesShannonOE}, dynamical entropy is an invariant of $\varphi$-integrable orbit equivalence. In particular it is an invariant of $\LL^p$ orbit equivalence for $p\in(0,+\infty)$.
		When $\varphi$ is moreover sublinear, this is the only invariant of $\varphi$-integrable orbit equivalence that we know for ergodic transformations, even for $\LL^p$ orbit equivalence where $p\in (0,1)$.
	\end{rmq}
	
	\section{Weakly mixing elements are generic in \texorpdfstring{$[T]_\ph$}{[T]phi}}\label{sec.weaklymixing}

	This last section is dedicated to the proof of Theorem \ref{thmIntro:belinskaya opti asym}: we are going to show that for every sublinear metric-compatible function $\ph$ and ergodic transformation $T$, there is an element $U\in [T]_\ph$ which has the same orbit as $T$ but is not flip-conjugate to $T$.

	Note that we have shown in Corollary \ref{cor.belinskayaoptimal} that this is already the case if $T$ is an ergodic transformation such that $T^n$ is ergodic for some $n\geq 2$. Therefore we will restrict ourselves to the case when there exists $n\geq 2$ such that $T^n$ is not ergodic. For such transformations, we will not construct any explicit $U\in [T]_\ph$, but we will use the Baire category theorem. We will show that given an aperiodic transformation $T$, the possible candidates of such $U$ are generic, see Theorem \ref{thm.gdelta dense weakly mixing}.

	We start with three preparatory sections to introduce the required material. We believe them to be of independent interest.
	In the first one, we show that the metric $\d_{\ph,T}$ is a complete separable metric inducing a Polish group topology on $[T]_\ph$, see Theorem \ref{thm.polishgrouptopology}. 
	In the second one, we prove a sublinear ergodic theorem in the context
	of $\varphi$-integrability which will play a crucial role later on.
	In the third one, we study continuity properties of the first return map.

	\subsection{Polish group topology}\label{sec.polish}
	Recall that the full group $[T]\leq\Aut(X,\mu)$ is closed and separable for the topology induced by the uniform metric $d_u$ and therefore it is a Polish group \cite[Prop.~3.2]{kechrisGlobalaspectsergodic2010}. We shall see that $\ph$-integrable full groups provide further interesting classes of Polish groups.

	Let $\varphi\colon \R_+\to\R_+$ be a metric-compatible function and let $T\in\Aut(X,\mu)$ be an aperiodic transformation. We introduced the $\ph$-integrable full group $[T]_\varphi$ as the group of measure-preserving transformation whose cocycle is $\ph$-integrable. In Lemma \ref{lem.distanceonphifullgroup}, we defined a metric $\d_{\varphi,T}$ on $[T]_\ph$. The goal of this section is to prove that the topology induced by $\d_{\ph,T}$ on $[T]_\ph$ is a Polish group topology.
	
	\begin{thm}\label{thm.polishgrouptopology}
		The metric $\d_{\ph,T}$ is complete, separable and right-invariant on $[T]_\ph$ and the topology generated by $\d_{\ph,T}$ is a group topology. In particular $[T]_\ph$ is a Polish group.
	\end{thm}
	\begin{proof}
		We have already shown in Lemma \ref{lem.distanceonphifullgroup} that $\d_{\ph,T}$ is right-invariant. Corollary \ref{cor: inclusions} tells us that the inclusion $[T]_\ph\into [T]$ is Lipschitz and in particular any $\d_{\ph,T}$-Cauchy sequence is $d_u$-Cauchy.
		Since $d_u$ is complete, any $\d_{\ph,T}$-Cauchy sequence has a $d_u$-limit. 
		\begin{claim*}\label{claim.1}
			Let $(U_n)_{n\geq 0}$ be a $\d_{\ph,T}$-Cauchy sequence of elements of $[T]_\ph$ and let $U\in [T]$ be its $d_u$-limit. Then $U\in [T]_\ph$ and $\lim_n\d_{\ph,T}(U_n,U)=0$. In particular $\d_{\ph,T}$ is complete.
		\end{claim*}
		\begin{cproof}
			Since $(U_n)_{n\geq 0}$ is $\d_{\ph,T}$-Cauchy, there is $m$ such that for all \mbox{$n\geq m$,} \[\int_X \lvert c_{U_n}(x)\rvert_\ph d\mu=\d_{\ph,T}(U_n,\id)\leq \d_{\ph,T}(U_n,U_m)+\d_{\ph,T}(U_m,\id)\leq 1+\d_{\ph,T}(U_m,\id).\]
			Moreover since $\lim_n d_u(U_n,U)=0$, we have that $(c_{U_n})_{n\geq 0}$ converges in measure to $c_U$ and thus a subsequence of $(c_{U_n})_{n\geq 0}$ converges pointwise to $c_U$. Fatou's lemma then implies that $U\in [T]_\ph$.
			The triangle inequality for $\lvert\cdot\rvert_\varphi$ gives
			\begin{align*}
				\int_X \big\lvert \lvert c_{U_n}(x)-c_U(x)\rvert_\ph-\lvert c_{U_m}(x)-c_U(x)\rvert_\ph\big\rvert d\mu&\leq
				\int_X \lvert c_{U_n}(x)-c_{U_m}(x)\rvert_\ph d\mu \\ &=\d_{\ph,T}(U_n,U_m),
			\end{align*}
			hence the sequence $(\lvert c_{U_n}-c_U\rvert_\ph)_{n\geq 0}$ is Cauchy with respect to the $\LL^1$-metric. Since $(c_{U_n}-c_U)_{n\geq 0}$ converges in measure to $0$, we must have that \[\lim_n\d_{\ph,T}(U_n,U)=\lim_n\int_X\lvert c_{U_n}(x)-c_U(x)\rvert_\varphi d\mu=0\]
			so the claim is proved.
		\end{cproof}
		Let us now show that 
		the topology induced by $\d_{\ph,T}$ is a group topology.
		We start by proving the continuity of the inverse map.
		Let $(U_n)_{n\geq 0}$ be a sequence of elements of $[T]_\ph$ converging to \mbox{$U\in [T]_\ph$}. Then the cocycle identity gives us that $0=c_{UU^{-1}}(x)=c_U(U^{-1}(x))+c_{U^{-1}}(x)$ and hence
		\begin{align*}
			\d_{\ph,T}(U_n^{-1},U^{-1})
			&=\int_X \lvert c_{U_n^{-1}}(x)-c_{U^{-1}}(x)\rvert_\ph d\mu\\
			&=\int_X \lvert c_{U_n}(U_n^{-1}(x))-c_{U}(U^{-1}(x))\rvert_\ph d\mu\\
			&=\int_X \abs{c_{U_n}(x)-c_U(U\inv U_n(x))}_\ph d\mu \\
			&\leq \int_X\lvert c_{U_n}(x)-c_U(x)\rvert_\ph d\mu+\int_X\lvert c_U(x)-c_U(U^{-1}U_n(x))\rvert_\ph d\mu.
		\end{align*}
		Since $(U_n)_{n\geq 0}$ converges to $U$ for the metric $\d_{\ph,T}$ and thus for the uniform metric, the right hand side converges to $0$ and hence the inverse map is continuous.
		
		We now prove that the multiplication map is continuous. Let $(U_n)_{n\geq 0}$ and $(V_m)_{m\geq 0}$  be two sequences which $\d_{\ph,T}$-converge to $U$ and $V$ respectively. Then by the triangle inequality and right-invariance,
		\[
		\d_{\ph,T}(U_nV_n,UV)\leq \d_{\ph,T}(U_n,U)+\d_{\ph,T}(UV_n,UV).
		\]
		Now remark that since the inverse map is continuous, $UV_n$ converges to $UV$ if and only if $V_n^{-1}U^{-1}$ converges to $V^{-1}U^{-1}$. By right-invariance and continuity of the inverse $\lim_n\d_{\ph,T}(V_n^{-1}U^{-1},V^{-1}U^{-1})=0$, which finishes the proof that $\d_{\ph,T}$ induces a group topology on $[T]_\ph$.
		
		We are left to show that this topology is separable. Consider the following abelian group  where we identify functions up to a null set\[\LL^\ph(X,\mathbb Z)\coloneqq\left\lbrace f\colon X\rightarrow \mathbb Z\colon \int_X\lvert f(x)\rvert_\varphi d\mu<+\infty\right\rbrace,\]
		endowed with the metric $(f,g)\mapsto\int_X\lvert f(x)-g(x)\rvert_\varphi d\mu$. 
		The function which takes $U\in[T]_\varphi$ to $c_U\in \LL^\ph(X,\mathbb Z)$ is an isometry. So $[T]_\ph$ is isometric to a metric subspace of $\LL^\ph(X,\mathbb Z)$.
		We now prove that $\LL^\ph(X,\mathbb Z)$ is separable: identify $X$ with $[0,1]$ equipped with the Lesbegue measure and observe that the subgroup generated by characteristic functions of rational intervals is dense. Since subspaces of separable metric spaces are separable, we conclude that $[T]_\ph$ is separable. 
	\end{proof}
	
	We now exploit the Polish group topology to characterize the inclusion between $\ph$-integrable full groups in terms of metric comparisons. In particular $[T]_\ph\neq [T]_\psi$ as soon as $\ph$ and $\psi$ are not bi-Lipschitz. However we do not know how to construct any explicit element in $[T]_\ph\setminus [T]_\psi$. 
	
	\begin{prop}\label{prop: chara equal phifull}
		Let $\varphi$ and $\psi$ be two metric-compatible functions and let $T\in\Aut(X,\mu)$ be an aperiodic measure-preserving transformation. Then the following are equivalent:
		\begin{enumerate}[(i)]
			\item \label{item: lip phi}
			$\ph(t)=O(\psi(t))$ as $t\to +\infty$.
			\item \label{item: inclusion}$[T]_\psi\leq [T]_\ph$.
		\end{enumerate}
	\end{prop}
	
	The proof uses the following well-known lemma.
	\begin{lem}\label{lem:AC}
		Let $G$ be a Polish group, let $H_1\leq H_2\leq G$ be two subgroups of $G$. Suppose that $H_1$ and $H_2$ are endowed with a Polish topology which refines the topology induced by $G$. Then the topology of $H_1$ refines the topology induced by $H_2$.
	\end{lem}
	\begin{proof}
		By hypothesis, the inclusions $H_1\into G$ and $H_2\into G$ are continuous. In particular, the Borel structure induced by each of their topologies refines the Borel structure induced by the one of $G$. 
		The Lusin-Souslin theorem states that given any two Polish spaces $X$ and $Y$, if $f\colon X\to Y$ is Borel and injective then for every Borel $A\subseteq X$, the set $f(A)$ is Borel, see \cite[Thm.~15.1]{kechrisClassicaldescriptiveset1995}. 
		Therefore, we can apply it to the inclusions $H_1\into G$ and $H_2\into G$ to obtain that the Borel structures induced by the respective topologies of $H_1$ and $H_2$ coincide with the $\sigma$-algebra induced by the Borel subsets of $G$.
		This in particular tells us that the inclusion map $H_1\into H_2$ is Borel, so it is automatically continuous by Pettis' lemma \cite[Thm.~9.9]{kechrisClassicaldescriptiveset1995} which proves the lemma.
	\end{proof}

	\begin{proof}[Proof of Proposition \ref{prop: chara equal phifull}]
		The implication \eqref{item: lip phi}$\Rightarrow$\eqref{item: inclusion} follows from Lemma \ref{lemma.comparisontopo}, so we only need to prove \eqref{item: inclusion}$\Rightarrow$\eqref{item: lip phi}. We argue by contradiction: assume that
		$[T]_\psi\leq [T]_\ph$ but \eqref{item: lip phi} does not hold.  By Lemma \ref{lem.lipschitzonN}, there exists a sequence $(k_n)_{n\geq 0}$ of positive integers such that
		\begin{equation}\label{eq:psi(k_n)/phi(k_n)}
			\lim_{n\to +\infty} \frac{\psi(k_n)}{\ph(k_n)}=0.
		\end{equation}
		Corollary \ref{cor: inclusions} tells us that  $[T]_\ph$ and $[T]_\psi$ embed continuously in $[T]$. Therefore Lemma \ref{lem:AC} yields that the inclusion map of $[T]_\psi$ into $[T]_\ph$ is continuous.
		We will obtain our contradiction by constructing a sequence $(U_n)_{n\geq 0}$ of elements of $[T]_\psi$ such that
		\[
		\d_{\psi,T}(U_n,\id)\to 0\text{ but }\d_{\ph,T}(U_n,\id)\not\to 0.
		\]
		
		By Rokhlin's lemma, one can find for every $n\in\N$, a measurable subset $A_n\subseteq X$ such that $A_n,T(A_n),\ldots, T^{2k_n-1}(A_n)$ are pairwise disjoint and $\mu(A_n)\geq \frac 1{4k_n}$. Note that 
		\[\mu\left(\bigsqcup_{i=0}^{k_n-1}T^i(A_n)\right)\geq \frac{1}{4}.\]
		Hence, for all $n$ such that $\ph(k_n)\geq 4$, we can pick a measurable subset $B_n\subseteq \bigsqcup_{i=0}^{k_n-1}T^i(A_n)$ of measure exactly $\frac 1{\ph(k_n)}$. We then define $U_n\in [T]_\ph$ by
		$$
		U_n(x)\coloneqq\left\{
		\begin{array}{ll}
			T^{k_n}(x)   &\text{if }x\in B_n;  \\
			T^{-k_n}(x)     &\text{if }x\in T^{k_n}(B_n);\\
			x&\text{otherwise.}
		\end{array}
		\right.
		$$
		By construction $\d_{\ph,T}(U_n,\id)=2\mu(B_n)\ph(k_n)=\frac 12$ but Equation \eqref{eq:psi(k_n)/phi(k_n)} implies that $d_{\psi,T}(U_n,\id)=2\mu(B_n)\psi(k_n) \to 0$, a contradiction.
	\end{proof}

	\begin{cor}
		Let $\varphi$ and $\psi$ be two metric-compatible functions, let $T\in\Aut(X,\mu)$ be aperiodic. Then $[T]_\varphi=[T]_\psi$ if and only if $\ph(t)=O(\psi(t))$ and $\psi(t)=O(\ph(t))$ as $t\to +\infty$.
	\end{cor}
	
	\subsection{A sublinear ergodic theorem for \texorpdfstring{$\varphi$}{phi}-integrable functions}\label{sec.subadditivethm}

	In this section, we prove the following sublinear ergodic theorem, which will be a key tool in our analysis of the first return map. 
	Given a measurable function $\varphi:\R_+\to\R_+$, a measurable function $f:X\to\C$ is \defin{$\varphi$-integrable} when 
	\(\int_X\varphi(\abs{f(x)})d\mu<+\infty\). 
	
	\begin{thm}\label{thm.subadditivephifullgroup}
		Let $\ph:\R_+\to\R_+$ be a sublinear metric-compatible function. 
		Let $U\in\Aut(X,\mu)$ and let $f\colon X\rightarrow \mathbb C$ be a measurable function which is $\ph$-integrable. 
		Then for almost every $x\in X$
		\[
		\lim_n \frac 1 n \varphi\left(\abs{\sum_{k=0}^{n-1} f(U^k(x))}\right)=0.
		\]
		The convergence also holds in $\LL^1$, that is
		\[
		\lim_n\int_X \frac 1 n\varphi\left(\left|\sum_{k=0}^{n-1} f(U^k(x))\right|\right) d\mu=0.
		\]
	\end{thm}	
	
\begin{proof}
		Given $n\geq 1$ and a $\varphi$-integrable function $f$, let for all $x\in X$
			\[ g_n(x)\coloneqq \varphi\left( \left\lvert\sum_{k=0}^{n-1}f(U^k(x))\right\rvert\right).\]

		Using the fact that $\varphi$ is metric-compatible, we deduce that the sequence of functions $(g_n)_{n\geq 1}$ satisfies Kingman's subadditivity property: for all $n,m\geq 1$ and all $x\in X$, 
		\[
		g_{n+m}(x)\leq g_n(x)+g_m(U^n(x)).
		\] Kingman's subadditive theorem \cite{kingmannErgodicTheorySubadditive} implies that $(\frac{g_n}n)_{n\geq 0}$ converges almost everywhere to some function $h$ and our aim becomes to show that $h=0$. Recall that a sequence that converges in $\LL^1$ admits an almost surely converging subsequence. In order to prove that $h=0$, it is therefore enough to prove that $\| \frac{g_n}n\|_1$ converges to $0$, namely to establish the second part of the theorem.

		To this end, let $f$ be a $\varphi$-integrable function and let $\varepsilon>0$. Since $\varphi(\abs f)$ is integrable and $\varphi$ is non-decreasing, we find a measurable subset
		$A\subseteq X$ and $K\geq 0$ such that $\int_{X\setminus A} \varphi(|f(x)|) d\mu\leq \eps$ and $\lvert f(x)\rvert\leq K$ for every $x\in A$. 
		For every measurable subset $B\subseteq X$, we denote $f_B\coloneqq f\mathds{1}_B$, where $\mathds1_B$ is the indicator function of $B$. With this notation at hand, using  first that $\varphi$ is subadditive non-decreasing and then that $U$ preserves the measure we obtain:
		\begin{align*}
			\limsup_{n}\int_X\frac 1 n\varphi\left(\left\lvert \sum_{k=0}^{n-1}f_{X\setminus A}(U^k(x))\right\rvert\right) d\mu &
			\leq \limsup_{n} \int_X\frac 1 n \sum_{k=0}^{n-1}\varphi\left(|f_{X\setminus A}(U^k(x))|\right) d\mu\\
			&= \int_{X}\varphi\left(\lvert f_{X\setminus A}(x)\rvert\right) d\mu\\ &\leq \eps.
		\end{align*}
		Besides, since $f_A$ is bounded by $K$, we have for all $x\in X$
		\[\frac{1}{n}\varphi\left(\abs{\sum_{k=0}^{n-1}f_A(U^k(x))}\right)\leq \frac{\varphi(nK)}{n}= K\frac{\varphi(nK)}{nK}.\] 
		Integrating over $X$, we obtain \[\int_X \frac 1 n\varphi\left(\left|\sum_{k=0}^{n-1} f_A(U^k(x))\right|\right) d\mu\leq K\frac{\varphi(nK)}{nK}.\]
		Using that $f=f_{X\setminus A}+f_{A}$ and subadditivity, we deduce
		\[
		\limsup_n\int_X \frac 1 n\ph\left(\left|\sum_{k=0}^{n-1} f(U^k(x))\right|\right) d\mu
		\leq\varepsilon+\limsup_n K\frac{\varphi(nK)}{nK}
		\]
		Since $\varphi$ is sublinear, we finally obtain
		\[\limsup_{n}\int_X \frac 1 n\varphi\left(\left|\sum_{k=0}^{n-1} f(U^k(x))\right|\right) d\mu\leq \varepsilon.\]
		This proves that  $\|\frac{g_n}n\|_1$ converges to $0$, thus ending the proof of the theorem.	
	\end{proof}

	Here is our main application, which will be a key tool in the following section. 
	
	\begin{cor}\label{cor.dphip tn}
		Let $\ph$ be a sublinear metric-compatible function and let $T\in\Aut(X,\mu)$ be an aperiodic transformation. Then for every $U\in [T]_\ph$,
		\[\lim_n\frac{\d_{\varphi,T}(U^n,\id)}{n}=0.\]
	\end{cor}
	\begin{proof}
		For all integers $n\geq 0$ and all $x\in X$, by the cocycle identity and the triangle inequality we have 
		\[\lvert c_{U^n}(x)\rvert\leq \sum_{k=0}^{n-1}\lvert c_U(U^k(x))\rvert.\]
		We apply Theorem \ref{thm.subadditivephifullgroup} to the function $f(x)\coloneqq \lvert c_U(x)\rvert$ and get that
		\[\frac{\d_{\varphi,T}(U^n,\id)}{n}\leq \int_X\frac{1}{n}\left\lvert\sum_{k=0}^{n-1}f(U^k(x))\right\rvert_\ph \underset{n\to +\infty}{\longrightarrow}0.\qedhere\]
	\end{proof}
	
	\begin{rmq} 
		\begin{sloppypar}
			We do not fully understand the asymptotics of the sequence $(\d_{\varphi,T}(U^n,\id))_{n\geq 0}$. For instance, when does the sequence $(\d_{\varphi,T}(U^n,\id)/\ph(n))_{n\geq 0}$ converge?
		\end{sloppypar}
	\end{rmq}
	
	\subsection{Continuity properties of the first return map}
	
	In the coming section we are primarily interested in continuity properties of the first return map.
	An important preliminary step is the following analogue of Kac's Lemma, saying that $\varphi$-integrable full groups are stable under first return maps.

	\begin{lem}\label{lem.kacphifullgroup}
		Let $T\in\Aut(X,\mu)$ be an aperiodic transformation and let \mbox{$\ph\colon  \R_+\to\R_+$} be a metric-compatible function. For all $U\in [T]_\ph$ and all measurable subsets $A\subseteq X$, we have 
		\(\d_{\varphi,T}(U_A,\id)\leq\d_{\varphi,T}(U,\id)\). In particular, $U_A\in [T]_\ph$.
	\end{lem}
	\begin{proof} 
		For every integer $j\geq 1$, set $A_j\coloneqq\{x\in A\colon n_{U,A}(x)=j\}$ where $n_{U,A}$ is the first return time of $U$ to $A$ as defined in Section \ref{sec.preliminaries}. Then, 
		\[
		\int_X \lvert c_{U_A}(x)\rvert_\ph d\mu=\int_A\abs{c_{U_A}(x)}_\ph d\mu=\sum_{j=1}^{+\infty}\int_{A_j} \lvert c_{U_A}(x)\rvert_\ph d\mu=\sum_{j=1}^{+\infty}\int_{A_j} \lvert c_{U^j}(x)\rvert_\ph d\mu
		\]
		By the cocycle identity, for every $j\geq 1$ we have $c_{U^j}(x)=\sum_{i=0}^{j-1}c_U(U^i(x))$, so by the triangle inequality we obtain
		\begin{align*}
			\int_X \lvert c_{U_A}(x)\rvert_\ph d\mu&\leq\sum_{j=1}^{+\infty}\sum_{i=0}^{j-1}\int_{A_j}\lvert c_U(U^i(x))\rvert_\ph d\mu\\
			&\leq\sum_{j=1}^{+\infty}\sum_{i=0}^{j-1}\int_{U^{i}(A_j)}\lvert c_U(x)\rvert_\ph d\mu\\
			&\leq\int_X \lvert c_U(x)\rvert_\ph d\mu,
		\end{align*} 
		the last inequality being a consequence of the fact that the sets $U^{i}(A_j)$ are disjoint for $j\in \N$ and $i\in\{0,\ldots,j-1\}$.
	\end{proof}
	
	In order to state the continuity properties of the first return map, let us first observe that since we are working up to measure zero, the first return map with respect to a set $A$ only depends on $A$ up to a null set.
	It is therefore natural to introduce the measure algebra $\MAlg(X,\mu)$, defined as the algebra of measurable subsets modulo identifying subsets which differ on a null set.  We endow $\MAlg(X,\mu)$ with the metric $d_\mu(A,B)\coloneqq \mu(A\bigtriangleup B)$. 
	
	We can now recall a continuity property satisfied by the first return map in the full group, which was first observed by Keane.
	
	\begin{lem}[{\cite[Lem.~3]{keaneContractibilityAutomorphismGroup1970}}]\label{lem: keane continuity}
		Let $T$ be a measure-preserving transformation, then the map
		\[
		\begin{array}{rl}
			[T]\times \MAlg(X,\mu)&\to[T]\\
			(U,A)&\mapsto U_A
		\end{array}
		\]
		is continuous.
	\end{lem}
	
	It is worth noting that the analogue of Lemma \ref{lem.kacphifullgroup} fails for the $\LL^1$-full group.
	Indeed, let $T\in\Aut(X,\mu)$ be ergodic and let $\ph\coloneqq \id_{\R_+}$ Then Kac's Lemma yields that for all measurable $A\subseteq X$ of positive measure, $\d_{\ph,T}(T_A,\id)=d_{\ph,T}(T,\id)=1$. Since $T_{\emptyset}=\id$, this shows that the map $\MAlg(X,\mu)\to [T]_1$ defined by $A\mapsto T_A$ is not continuous.
	
	However, the situation is not that clear when $\varphi$ is sublinear.
	
	\begin{qu}
		Let $\ph\colon\R_+\to\R_+$ be a \emph{sublinear} metric-compatible function. Is the map $\MAlg(X,\mu)\to [T]_\ph$ defined by $A\mapsto T_A$ continuous? More generally, is the map $[T]_\ph\times \MAlg(X,\mu)\to [T]_\ph$ given by $(U,A)\mapsto U_A$ continuous? 
	\end{qu}
	
	In this section we give two partial answers to the above questions.
	We first prove that the map $A\mapsto U_A$ satisfies a continuity property ``from below''. For this, we need the following version of Scheffé's lemma for sequences of $\Z$-valued $\varphi$-integrable functions.
	
	\begin{lem}\label{lem.scheffe}
		Let $f\colon X\to\Z$ be a measurable function and let $(f_n)_{n\geq 0}$ be a sequence of measurable functions $f_n\colon X\to\Z$ that converge in measure to $f$. If
		$$\limsup \int_X\abs{f_n}_\ph d\mu\leq \int_X \abs f_\ph d\mu,$$
		then $\lim_n\int_X\abs{f_n-f}_\ph d\mu=0$. 
	\end{lem}
	\begin{proof}
		It suffices to show that given $\varepsilon>0$, there is $\delta>0$ such that for all measurable functions $g\colon X\to\Z$ satisfying
		\begin{equation}\label{eq: hyp scheffé}
			\mu(\{x\in X\colon f(x)\neq g(x)\})\leq\delta\text{ and }\int_X\abs{g}_\ph d\mu\leq \int_X\abs f_\ph d\mu+\delta,
		\end{equation}
		one has that \(\int_X\abs{f-g}_\ph d\mu\leq\varepsilon\).
		To this end, fix $\varepsilon>0$. Since $\int_X\lvert f\rvert_\ph d\mu<+\infty$, by Lebesgue's dominated convergence theorem there exists $\delta_0>0$ such that for all measurable subsets $A\subseteq X$, if $\mu(A)<\delta_0$ then $\int_A\lvert f\rvert_\ph d\mu<\varepsilon$. 
		Take $\delta \coloneqq\min\{\delta_0,\varepsilon\}$.
		Let $g\colon X\to\Z$ be a measurable function satisfying \eqref{eq: hyp scheffé}. If we let $A\coloneqq \{x\in X\colon f(x)\neq g(x)\}$, we have
		\[\int_A \lvert g\rvert_\ph d\mu =\int_X \lvert g\rvert_\ph d\mu-\int_{X\setminus A}\lvert g\rvert_\ph d\mu\leq \int_X \lvert f\rvert_\ph d\mu-\int_{X\setminus A}\lvert f\rvert_\ph d\mu +\delta\leq 2\eps\]
		and we can therefore conclude the proof
		\[\int_X\abs{f-g}_\ph d\mu=\int_A \abs{f-g}_\ph d\mu\leq \int_A\abs f_\ph d\mu+\int_A \abs g_\ph d\mu\leq 3\varepsilon.\qedhere\]
	\end{proof}
	We can now prove the following proposition which is the $\varphi$-integrable analogue of \cite[Prop.~3.9]{lemaitremeasurableanaloguesmall2018}.
	
	\begin{prop}\label{prop.continuitedroiteinduite} 
		Let $\ph$ be a metric-compatible function and $T\in\Aut(X,\mu)$ an ergodic transformation. 
		Consider $U\in [T]_\varphi$ and consider a measurable subset $A\subseteq X$. If $(A_n)_{n\geq 0}$ is a sequence of measurable subsets of $A$ such that $\lim_n\mu(A\setminus A_n)= 0$, then 
		\(\lim_n\d_{\varphi,T}(U_{A_n},U_A)=0.\)
	\end{prop}
	\begin{proof}
		Since $\lim_n\mu(A\setminus A_n)=0$ and since the first return map is continuous with respect to the uniform metric by Lemma \ref{lem: keane continuity}, we get that $\lim_n d_u(U_{A_n},U_A)=0$. 
		This means that $(c_{U_{A_n}})_{n\geq 0}$ converges in measure to $c_{U_A}$ and therefore that $(\lvert c_{U_{A_n}}\rvert_\ph)_{n\geq 0}$ converges in measure to $\lvert c_{U_A}\rvert_\ph$. 
		Thanks to Lemma \ref{lem.kacphifullgroup}, we have for all $n\geq 0$, $d_{\ph,T}(U_{A_n},\id)\leq d_{\ph,T}(U_A,\id)$. In other words
		\[\int_X\lvert c_{U_{A_n}}(x)\rvert_\ph d\mu \leq \int_X\lvert c_{U_{A}}(x)\rvert_\ph d\mu.\] 
		Hence we can apply Lemma \ref{lem.scheffe}, yielding \[\lim_n\int_X \lvert c_{U_{A_n}}(x)-c_{U_A}(x)\rvert_\ph d\mu= 0.\] This precisely means that $\lim_n \d_{\varphi,T}(U_{A_n},U_A)= 0$, so we are done.
	\end{proof}

	Let $\varphi \colon  \R_+\to\R_+$ be a sublinear metric-compatible function and $T\in\Aut(X,\mu)$ be an aperiodic transformation. In Corollary \ref{cor.distanceinduiteaid}, we proved that for any aperiodic transformation $T\in\Aut(X,\mu)$, the quantity $\d_{\ph,T}(T_A,\id)$ tends to $0$ as $\mu(A)$ approaches $0$.
	It is natural to ask whether this holds for all aperiodic $U\in [T]_\varphi$, i.e.\ does $\d_{\ph,T}(U_A,\id)$ tends to $0$ as $\mu(A)$ approaches $0$? We were not able to answer this question, but we can prove the following much weaker statement. Its proof relies on our sublinear ergodic theorem (Theorem \ref{thm.subadditivephifullgroup}), or rather on Corollary \ref{cor.dphip tn}.
	
	\begin{prop}\label{prop. cheap return}
		Let $\ph$ be a sublinear metric-compatible function. Let $T\in\Aut(X,\mu)$ be an aperiodic transformation. Then for any aperiodic transformation $U\in [T]_\ph$ and for any measurable subset $A\subseteq X$, there exists a sequence $(A_k)_{k\geq 0}$ of measurable subsets contained in $A$ which intersect every $U$-orbit, such that $\lim_k\mu(A_k)=0$ and $\lim_k\d_{\varphi,T}(U_{A_k},\id)=0.$
	\end{prop}
	\begin{proof}
		Put $V\coloneqq U_A$ and remark that for every measurable $B\subseteq A$, we have that $V_B=U_B$. 
		As an immediate consequence of Alpern's multiple Rokhlin theorem \cite{alpernGenericPropertiesMeasure1979}\footnote{We actually only need Step 1 from the simpler proof given in \cite{eigenMultipleRokhlinTower1997}.}, for every $k\geq 0$, one can find a measurable subset $B_k\subseteq A$ which meets every $V$-orbit in $A$ and such that $n_{V,B_k}(B_k)=\{k,k+1\}$. 
		The latter implies that the $V^i(B_k)$ are disjoint for $i\in\{0,...,k-1\}$. Observe that for all $x\in X$ and $i\in\Z$ we have $n_{V,V^i(B_k)}(x)=n_{V,B_k}(V^{-i}(x))$. This implies that for all $x\in V^i(B_k)$, either $V_{V^i(B_k)}(x)=V^k(x)$ or $V_{V^i(B_k)}(x)=V^{k+1}(x)$. Therefore by integrating over the disjoint union of the $V^i(B_k)$ for $i\in\{0,...,k-1\}$ 
		we get that \[\sum_{i=0}^{k-1}\d_{\varphi,T}(V_{V^i(B_k)},\id)\leq \d_{\varphi,T}(V^{k},\id)+\d_{\varphi,T}(V^{k+1},\id),\]
		whence there exists $0\leq i_k\leq k-1$ such that 
		\[\d_{\varphi,T}(V_{V^{i_k}(B_k)},\id)\leq \frac{\d_{\varphi,T}(V^{k},\id)+\d_{\varphi,T}(V^{k+1},\id)}{k}.\]
		The set $A_k\coloneqq V^{i_k}(B_k)$ has measure less than $1/k$. Corollary \ref{cor.dphip tn} implies that the right hand side in the above formula tends to zero, which implies \[\lim_k\d_{\varphi,T}(U_{A_k},\id)=\lim_k\d_{\ph,T}(V_{A_k},\id)=0.\qedhere\]
	\end{proof}
	
	\subsection{Optimality of Belinskaya's theorem}
	
	We are now ready to prove Theorem \ref{thmIntro:belinskaya opti asym}: for any sublinear function $\ph$, Belinskaya's theorem fails if we replace integrability by $\varphi$-integrability.
	
	\begin{thm}\label{thm:genericcounterexamples}
		Let $\varphi\colon\R_+\to \R_+$ be a sublinear function and let $T_1\in\Aut(X,\mu)$ be ergodic. Then there exists an ergodic transformation $T_2\in [T_1]$ whose cocycle is $\ph$-integrable such that $T_1$ and $T_2$ have the same orbits but are not flip-conjugate.
	\end{thm}
	
	The proof of the theorem depends on whether $T_1$ is weakly mixing. Indeed if this is the case, then we can use Corollary \ref{cor.belinskayaoptimal}. Otherwise, we have to use the Baire category theorem. Indeed the candidate for $T_2$ is generic for the topology induced by $\d_{\ph,T_1}$.
	
	\begin{thm}\label{thm.gdelta dense weakly mixing}
		Let $\ph$ be a sublinear metric-compatible function and let $T\in\Aut(X,\mu)$ be an aperiodic element. Then the set of all elements of $[T]_\ph$ which are weakly mixing and have the same orbits as $T$ is a dense $G_\delta$ set in the Polish space of aperiodic elements of $[T]_\ph$ with respect to the topology induced by $\d_{\ph,T}$.
	\end{thm}
	
	We delay the proof of the above theorem to Section \ref{sub.aper}. Let us first explain how to deduce Theorem \ref{thm:genericcounterexamples} from Theorem \ref{thm.gdelta dense weakly mixing}. 
	
	\begin{proof}[Proof of Theorem \ref{thm:genericcounterexamples}]
		By Lemma \ref{lem.free metric-compatible}, there is a sublinear metric-compatible function $\psi$ such that $\ph(t)\leq \psi(t)$ for all $t$ large enough. 
		In particular, $\psi$-integrability implies $\ph$-integrability for $\Z$-valued functions (cf.\ Remark \ref{rmk.restriction to metric-compatible}). Hence if the theorem holds for $\psi$ then it holds for $\ph$.
		Therefore, by replacing $\ph$ by $\psi$, we may and do assume that $\ph$ is a metric-compatible function. 
		
		If $T_1$ is weakly mixing, then all its nontrivial power are ergodic. Thus Corollary \ref{cor.belinskayaoptimal} implies that there exists $T_2\in [T_1]_\ph$ such that $T_1$ and $T_2$ have the same orbits but are not flip-conjugate.
		
		If $T_1$ is not weakly mixing, then Theorem \ref{thm.gdelta dense weakly mixing} grant us the existence of some weakly mixing $T_2\in [T_1]_\ph$ such that $T_1$ and $T_2$ have the same orbits. Since $T_2$ is weakly mixing and $T_1$ isn't, they cannot be flip-conjugate.
	\end{proof}
	
	\subsection{Weakly mixing elements form a dense \texorpdfstring{$G_\delta$}{G delta} set}\label{sub.aper}
	
	This section is dedicated to the proof of Theorem \ref{thm.gdelta dense weakly mixing}. Before starting the proof, we will need some terminology and preliminary propositions. 
	
	In this section we will consider the $\ph$-integrable full groups both with the topology induced by the uniform metric $d_u$ and the their natural topology induced by $\d_{\ph,T}$. The metric $\d_{\ph,T}$ is complete so we can apply the Baire category theorem in $([T]_\varphi,\d_{\ph,T})$, see Theorem \ref{thm.polishgrouptopology}. Moreover, the topology induced by $\d_{\varphi,T}$ refines the topology induced by $d_u$, see Corollary \ref{cor: inclusions}. 
	Note that $([T]_\varphi,d_u)$ is not complete, indeed one can show that $[T]_\ph$ is dense in the complete metric space $([T],d_u)$.
	
	Denote by $\APER\subseteq\Aut(X,\mu)$ the set of aperiodic transformations.
	
	\begin{lem}
		Let $\ph$ be a metric-compatible function and let $T\in\Aut(X,\mu)$ be an aperiodic element.
		Then the set $\APER\cap [T]_\ph$ is closed in the complete metric space $([T]_\varphi,\d_{\ph,T})$ and hence it is a complete metric space itself.
	\end{lem}
	\begin{proof}
		Note that $T$ is aperiodic if and only if for all $n\geq 1$ we have $d_u(T^n,\id)=1$. So the set $\APER$ is closed in $(\Aut(X,\mu),d_u)$. 
		In particular, $\APER\cap [T]_\varphi$ is closed in $([T]_\varphi,d_u)$,
		so it is also closed in $([T]_\varphi,\d_{\ph,T})$.
	\end{proof}
	
	\begin{prop}\label{prop.memeorbitesGdeltadense} 
		Let $\ph$ be a sublinear metric-compatible function and consider an aperiodic element $T\in\Aut(X,\mu)$. Then the set 
		\[\{U\in\APER\cap [T]_\varphi\colon T\text{ and }U\text{ have the same orbits}\}\]
		is a dense $G_\delta$ set in $(\APER\cap[T]_\varphi,\d_{\ph,T})$.
	\end{prop}
	\begin{proof}
		We first prove that this set is $G_\delta$. For all $\varepsilon >0$ and $n\geq 1$, let 
		\[O_{\varepsilon,n}\coloneqq \left\{U\in [T]_\varphi\colon \mu(\{x\in X\colon T(x)\in\{U^{-n}(x),\dots,U^n(x)\}\})>1-\varepsilon\right\}.\]
		Each $O_{\varepsilon,n}$ is open in $([T]_\varphi,d_u)$ and thus also in $([T]_\ph,\d_{\ph,T})$. Moreover, we have
		\[\{U\in\APER\cap [T]_\varphi\colon T\text{ and }U\text{ have the same orbits}\}=\bigcap_{\varepsilon\in\Q_+^*}\bigcup_{n\geq 1}O_{\varepsilon,n},\]
		which is a countable intersection of open sets, thus by definition a $G_\delta$ set. 
		
		We now prove the density. Let $U\in\APER\cap [T]_\ph$. Fix a sequence $(A_k)_{k\geq 0}$ of measurable subsets of $X$ which intersect every $S$-orbit, such that $\lim_k\mu(A_k)=0$ and $\lim_k\d_{\ph,T}(U_{A_k},\id)=0$ as in Proposition \ref{prop. cheap return}. If we set $P_k\coloneqq(U_{A_k})^{-1}U$, then we get that $(P_k)_{k\geq 0}$ tends to $U$. Moreover, Corollary \ref{cor.distanceinduiteaid} implies that $(T_{A_k})_{k\geq 0}$ tends to the identity, which implies that $(T_{A_k}P_k)_{k\geq 0}$ tends to $U$. On the other hand, Lemma \ref{lem.P is Periodic} yields that the transformation $P_k$ is periodic and $A_k$ is a fundamental domain for it. Thus, the transformations $T_{A_k}P_k$ and $T$ have the same orbits by Lemma \ref{lem.sameorbits} and the proof is completed.
	\end{proof}
	
	Let $\ERG$ denote the set of ergodic transformations in $\Aut(X,\mu)$. 
	
	\begin{prop}\label{prop: ergo is dense in aper} 
		Let $\ph$ be a sublinear metric-compatible function and fix an ergodic transformation $T\in \Aut(X,\mu)$. Then $\ERG\cap [T]_\varphi$ is a dense $G_\delta$ set in $(\APER\cap[T]_\varphi,\d_{\ph,T})$. 
	\end{prop}
	\begin{proof}
		By \cite[Thm.~3.6]{kechrisGlobalaspectsergodic2010}, $\ERG\cap [T]$ is a $G_\delta$ set in $(\APER\cap [T],d_u)$.
		Thus $\ERG\cap [T]_\ph$ is a $G_\delta$ set in $(\APER\cap [T]_\ph,\d_{\ph,T})$. Finally, since $T$ is ergodic, 
		\[\ERG\cap [T]_\varphi\supseteq\{S\in\APER\cap [T]_\varphi\colon S \text{ and }T\text{ have the same orbits}\}.\] 
		Thus, Proposition \ref{prop.memeorbitesGdeltadense} yields that $\ERG\cap [T]_\varphi$ is dense in $(\APER\cap [T]_\ph,\d_{\ph,T})$ and the proof is complete.
	\end{proof}
	
	\begin{rmq} 
		The hypothesis that $\varphi$ is sublinear is necessary, as for any ergodic $T\in \Aut(X,\mu)$, we have that $\ERG\cap [T]_1$ is not dense in $\APER\cap [T]_1$. Indeed, one can define a continuous \emph{index map} $I\colon [T]_1\to\Z$ by integrating the cocycle (see \cite[Cor.~4.20]{lemaitremeasurableanaloguesmall2018} for the fact that it takes values in $\Z$). Then note that $I(U)\neq 0$ for every ergodic $U\in [T]_1$: 
		by \cite[Prop.~4.13]{lemaitremeasurableanaloguesmall2018} every ergodic $U\in[T]_1$ is either almost positive or almost negative. Then combining \cite[Prop.~4.17 and Prop.~3.4]{lemaitremeasurableanaloguesmall2018} yields that $U$ has positive or negative index, so $I(U)\neq 0$. Finally, there are aperiodic elements with index $0$: take $A\subseteq X$ measurable with $0<\mu(A)<1$, then the aperiodic transformation $U\coloneqq T_AT_{X\setminus A}\inv$ has index zero (using again \cite[Prop.~3.4]{lemaitremeasurableanaloguesmall2018}).
		By continuity of the discrete-valued index map, we conclude that $\ERG\cap [T]_1$ cannot be dense in $\APER\cap [T]_1$.
	\end{rmq}
	
	\begin{df}
		A transformation $S\in\Aut(X,\mu)$ is \defin{weakly mixing} if for all finite subsets $\mathcal F\subseteq\MAlg(X,\mu)$ and all $\varepsilon >0$, there exists $n\in\Z$ such that for all $A,B\in\mathcal{F}$,
		\[\lvert \mu(V^n(A)\cap B)-\mu(A)\mu(B)\rvert <\varepsilon.\]
	\end{df}
	
	Given a measurable subset of positive measure $A\subseteq X$ we will denote by $\mu_A$ the probability measure on $A$ defined by $\mu_A(B)\coloneqq \mu(A\cap B)/\mu(A)$. We say that a transformation $T\in\Aut(X,\mu)$ is \textit{weakly mixing on} $A$ if $T(A)=A$ and the restriction of $T$ to $A$ is weakly mixing as an element of $\Aut(A,\mu_A)$.
	The following result will be crucial in the proof of Theorem \ref{thm.gdelta dense weakly mixing}.
	
	\begin{thm}[Conze \cite{conzeEquationsFonctionnellesSystemes1972}]
		Let $T\in\Aut(X,\mu)$ be an ergodic transformation. Then the set \[\{A\in\MAlg(X,\mu)\colon T_A\text{ is weakly mixing on }A\}\] is a dense $G_\delta$ set in $(\MAlg(X,\mu),d_\mu)$ where $d_\mu(A,B)\coloneqq\mu(A\bigtriangleup B)$.
	\end{thm}
	
	Denote by $\WMIX$ the set of weakly mixing transformations of $\Aut(X,\mu)$.
	We are finally ready to prove Theorem \ref{thm.gdelta dense weakly mixing} which can be reformulated as follows. 
	
	\begin{thm} 
		Let $\ph$ be a sublinear metric-compatible function and let $T\in\Aut(X,\mu)$ be an ergodic transformation. Then the set
		\[\{U\in \WMIX\cap[T]_\varphi\colon T\text{ and }U\text{ have the same orbits}\}\] is a dense $G_\delta$ set in $(\APER\cap [T]_\ph,\d_{\ph,T})$.
	\end{thm}
	\begin{proof}
		By the Baire category theorem, the intersection of two dense $G_\delta$ subsets is a dense $G_\delta$ subset.
		Hence by Proposition \ref{prop.memeorbitesGdeltadense}, it suffices to show that
		$\WMIX\cap [T]_\varphi$
		is a dense $G_\delta$ set in $(\APER\cap[T]_\varphi,\d_{\ph,T})$, which will occupy the remainder of the proof.
		
		By definition, a transformation $U$ is weakly mixing if and only if
		for all finite subsets $\mathcal F\subseteq\MAlg(X,\mu)$ and all $\varepsilon >0$,
		it belongs to the set $O_{\mathcal F,\varepsilon}$ defined by:
		\[O_{\mathcal F,\varepsilon}\coloneqq \left\{V\in\Aut(X,\mu)\colon \exists n\in\Z,\ \forall A,B\in \mathcal F,\ \lvert \mu(V^n(A)\cap B)-\mu(A)\mu(B)\rvert <\varepsilon\right\}.\]
		Observe that each $O_{\mathcal F,\varepsilon}$ is open in $(\Aut(X,\mu),d_u)$. As before, denote by $d_\mu$ the metric on $\MAlg(X,\mu)$ defined by $d_\mu(A,B)=\mu(A\bigtriangleup B)$. 
		
		\begin{claim}\label{claim.toomany lines of inequalities}
			Let $\mathcal F=\{A_1,...,A_n\}$ and $\mathcal F'=\{A'_1,...,A'_n\}$ be subsets of $\MAlg(X,\mu)$. Fix $\eps>0$. If for every $i\in\{1,\ldots,n\}$ one has  $\mu(A_i\bigtriangleup A'_i)<\eps$, then
			$$O_{\mathcal F,\eps}\subseteq O_{\mathcal F',5\eps}.$$
		\end{claim}
		\begin{cproof}
			Let $V\in O_{\mathcal F,\eps}$. Fix $n\in\Z$ such that for all $i,j\in\{1,\dots,n\}$ we have $\lvert \mu(V^n(A_i)\cap A_j)-\mu(A_i)\mu(A_j)\rvert<\eps$. We first remark that for every measurable $B\subset X$ and $i\in\{1,\ldots,n\}$, we have $\lvert\mu(B)\mu(A_i)-\mu(B)\mu(A_i')\rvert<\eps$ and $\lvert \mu(B\cap A_i')-\mu(B\cap A_i)\rvert<\eps$. The result now follows from the triangle inequality and the fact $V$ preserves $\mu$:
			\begin{align*}
				\lvert \mu(V^n(A_i')\cap A_j')-\mu(A_i')\mu(A_j')\rvert 
				&< \lvert \mu(V^n(A_i')\cap A_j')-\mu(A'_i)\mu(A_j)\rvert +\eps\\
				&< \lvert \mu(V^n(A_i')\cap A_j')-\mu(A_i)\mu(A_j)\rvert +2\eps\\
				&< \lvert \mu(V^n(A_i')\cap A_j)-\mu(A_i)\mu(A_j)\rvert +3\eps\\
				&=\lvert \mu(A_i'\cap V^{-n}(A_j))-\mu(A_i)\mu(A_j)\rvert +3\eps\\
				&<\lvert \mu(A_i\cap V^{-n}(A_j))-\mu(A_i)\mu(A_j)\rvert +4\eps\\
				&= \lvert \mu(V^n(A_i)\cap A_j)-\mu(A_i)\mu(A_j)\rvert +4\eps\\
				&< 5\eps.\qedhere
			\end{align*}
		\end{cproof}
		
		Since $(X,\mu)$ is standard, we can fix a countable dense subset $\mathcal M$ of $(\MAlg(X,\mu),d_\mu)$. 
		It follows from the Claim \ref{claim.toomany lines of inequalities} that 
		\begin{equation}\label{eq.WMIXGdelta}\WMIX = \bigcap_{\varepsilon \in\Q_+^*}\bigcap_{\mathcal F\subseteq \mathcal M \text{ finite}}O_{\mathcal F,\varepsilon}.\end{equation}
		In particular $\WMIX$ is a $G_\delta$ set in $(\Aut(X,\mu),d_u)$ and hence $\WMIX\cap [T]_\ph$ is a $G_\delta$ set in $([T]_\ph,\d_{\ph,T})$.
		
		We now prove that $\WMIX$ is dense. By the Baire category theorem, it is enough to show that each $O_{\mathcal{F},\eps}$ is dense in $(\APER\cap [T]_\ph,\d_{\ph,T})$. By Proposition \ref{prop: ergo is dense in aper} the set $\ERG\cap [T]_\varphi$ is dense in $\APER\cap [T]_\ph$, so it is enough to prove that 
		\begin{equation}\label{eq.aim of end of proof}
			\ERG\cap [T]_\ph\subseteq \overline{O_{\mathcal F,\eps}\cap \APER\cap [T]_\ph}.
		\end{equation}
		
		So let us fix a finite subset $\mathcal F\subseteq \MAlg(X,\mu)$, a positive real $\eps>0$ and an ergodic transformation $U\in \ERG\cap [T]_\ph$. 
		
		We let $(X_k)_{k\geq 0}$ be a sequence of measurable subsets such that $ \mu(X_k)=1- 2^{-k}$. For all $k\geq 0$, we apply Conze's Theorem to the transformation $U_{X_k}$, which is ergodic on $X_k$, to find a measurable subset $Y_k\subseteq X_k$ such that $\mu(Y_k)>1-2^{-k+1}$ and $U_{Y_k}$ is weakly mixing on $Y_k$.
		Set $V_k\coloneqq U_{Y_k}T_{X\setminus Y_k}$. We claim that $(V_k)_{k\geq 0}$ tends to $U$. Indeed since $\lim_k\mu(Y_k)=1$, Proposition \ref{prop.continuitedroiteinduite} yields that $U_{Y_k}$ tends to $U$ while 
		Corollary \ref{cor.distanceinduiteaid} gives us that $T_{X\setminus Y_k}$ tends to the identity. 
		
		\begin{claim}\label{claim. victory}
			For $k$ large enough, we have that $V_k\in O_{\mathcal F,\eps}$.
		\end{claim}
		\begin{cproof}
			For all $k\geq 0$, put $\mathcal{F}_k\coloneqq\{A\cap Y_k\colon A\in\mathcal{F}\}$. 
			Since $U_{Y_k}$ is weakly mixing on $Y_k$, we have $U_{Y_k}\in O_{\mathcal{F}_k,\eps/5}$. 
			By construction, the transformations $U_{Y_k}$ and $V_k$ coincide on $Y_k$, so we also have $V_k\in O_{\mathcal{F}_k,\eps/5}$.
			Since $\lim_k\mu(Y_k)=1$, for $k$ large enough and all $A\in\mathcal{F}$, we have $\mu(A\bigtriangleup (A\cap Y_k))<\eps/5$. 
			We thus get that $V_k\in O_{\mathcal{F},\eps}$ for $k$ large enough by Claim \ref{claim.toomany lines of inequalities}.
		\end{cproof}
		
		It follows immediately from Claim \ref{claim. victory} that any ergodic element in $[T]_\ph$ is a limit of aperiodic elements in $O_{\mathcal F,\eps}\cap [T]_\ph$. This shows the inclusion \eqref{eq.aim of end of proof}, ending the proof of the theorem.
	\end{proof}

	\appendix
	\section{Proof of Belinskaya's theorem}
	
	In this appendix, we present a short proof of Belinskaya's theorem due to Katznelson which is not publicly available to our knowledge \cite{katznelsonLecturesOrbitEquivalence1980}. 
	As in Belinskaya's original proof, a key step is the following theorem, of independent interest. 
    To lighten notation, given a point $x\in X$, a map $T\colon X\to X$, and a subset $I\subseteq \Z$, we will write
	\[ T^I(x)\coloneqq \{T^i(x)\colon i\in I\}.\]
	
	\begin{thm}[Katznelson]\label{thm: belinskaya commensurates}
		Let $T$ be an aperiodic measure-preserving transformation, suppose $U\in \Aut(X,\mu)$ has the same orbits as $T$ and that for almost every $x\in X$, the symmetric difference of the respective positive $T$ and $U$ orbits \(T^{\N}(x)\bigtriangleup  U^\N(x)\) is finite. 
		Then $T$ and $U$ are conjugate.
	\end{thm}

	\begin{proof}
		We will explicitly define an element $S$ in $[T]$ such that  $U=S^{-1}TS$. This will be done thanks to the following claim.
		\begin{claim*}\label{claim.defS}
			For almost every $x\in X$, there exists a unique $j(x)\in\Z$ such that \[\left\lvert  T^{\N+j(x)}(x)\setminus U^{\N}(x)\right\rvert=\left\lvert  U^\N(x)\setminus T^{\N+j(x)}(x)\right\rvert\]
		\end{claim*}
		\begin{cproof}
			For almost every $x\in X$, consider the function $\tau_x\colon\Z\to\Z$ defined by
			\[ \tau_x(j)\coloneqq \left\lvert  T^{\N+j}(x)\setminus U^{\N}(x)\right\rvert-\left\lvert  U^\N(x)\setminus T^{\N+j}(x)\right\rvert.\]
			Remark that by assumption for almost every $x$, the value $\tau_x(j)$ is finite for all $j\in\mathbb Z$. By considering the two cases $T^j(x)\in U^\N(x)$ and $T^j(x)\not\in U^\N(x)$, we see that $\tau_x(j+1)=\tau_x(j)-1$ for all $j\in\Z$. 
			It follows that $\tau_x(j)=\tau_x(0)-j$ for all $j\in\Z$ so $j(x)\coloneqq\tau_x(0)$ is the unique element we seek.
		\end{cproof}
		
		We set $S(x)\coloneqq T^{j(x)}(x)$. By the above claim, $S(x)$ is the unique element of the $T$-orbit of $x$ satisfying 
		\begin{equation}\label{eq.S(x)definition}
			\abs{T^\N(S(x))\setminus U^\N(x)}=\abs{U^\N(x)\setminus  T^\N(S(x))}.
		\end{equation}
		By considering whether none, only one, or both of the points $x$ and $S(x)$ belong to $T^{\N}(S(x))\cap U^\N(x)$,
		we see that removing the point $S(x)$ from $T^{\N}(S(x))$ and the point $x$ from $U^\N(x)$ does not perturb the above equation, so that  
		\[\abs{T^{\N+1}(S(x))\setminus U^{\N+1}(x)}=
		\abs{U^{\N+1}(x)\setminus  T^{\N+1}(S(x))}.\]
		This can be rewritten as
		\[\left|T^\N(TS(x))\setminus U^\N(U(x))\right|=\left|U^\N(U(x))\setminus T^\N(TS(x))\right|,\]
		which by equation \eqref{eq.S(x)definition} yields the desired equivariance condition
		\begin{equation*}\label{eq:intertwine}	SU(x)=TS(x).
		\end{equation*}
		We now have to check that $S\in [T]$.
		Using that $T$ and $U$ are invertible and a straightforward induction, we obtain that $SU^n(x)=T^nS(x)$ for all $n\in\Z$. In particular $S$ induces a bijection from the $T$-orbit of $x$ to the $U$-orbit of $S(x)$. Since $S(x)=T^{j(x)}(x)$ belongs to the $T$-orbit of $x$, which coincides with the $U$-orbit of $x$, we conclude that $S$ induces a bijection on each $T$-orbit, in particular $S$ is bijective. Finally we check that $S$
		is measure-preserving. The sets $A_n\coloneqq \{ x\in X\colon S(x)=T^n(x)\}$ for $n\in\Z$ form a partition of $X$. If $B\subseteq X$ is measurable, we write $B=\bigsqcup_n A_n\cap B$ so that
		$\mu(S(B))=\sum_n\mu(T^n(A_n\cap B))=\sum_n\mu(A_n\cap B)=\mu(B)$. This ends the proof of Theorem \ref{thm: belinskaya commensurates}.
	\end{proof}
	
	Given $T\in \Aut(X,\mu)$, denote by $\mathcal R_T\subseteq X\times X$ the equivalence relation whose classes are the $T$-orbits.
	Before proceeding with the proof of Belinskaya's theorem, we recall the following well-known lemma. Its usefulness towards proving Belinskaya's theorem was pointed out to us by Todor Tsankov.
	\begin{lem}[Mass-transport principle]\label{lem:MassTransport}  
		Let $T\in \Aut(X,\mu)$ and $f\colon \mathcal R_T\to \N$ be a measurable map. Then
		\[ \int_X\sum_{n\in\Z} f(x,T^n(x))d\mu=\int_X\sum_{n\in\Z} f(T^n(x),x)d\mu.\]   
	\end{lem}  
	\begin{proof}    
		Since $f$ is non-negative, Tonelli's theorem tells us that
		\begin{align*}
			\int_X\sum_{n\in\Z} f(x,T^n(x))d\mu&=\sum_{n\in\Z}\int_Xf(x,T^n(x))d\mu\\
			&=\sum_{n\in\Z}\int_Xf(T^{-n}(x),x)d\mu\\
			&=\sum_{n\in\Z}\int_X f(T^n(x),x)d\mu\\
			&=\int_X\sum_{n\in\Z} f(T^n(x),x)d\mu.
			\qedhere\end{align*} 
	\end{proof}

	\begin{thm}[Belinskaya's theorem]
		Let $T\in \Aut(X,\mu)$ be ergodic, let $U\in [T]_1$ have the same orbits as $T$. Then $T$ and $U$ are flip-conjugate.
	\end{thm}
	\begin{proof}
		Define a $T$-invariant total order $\leq_T$ on each $T$-orbit by setting $x\leq_Ty$ if there is $n\geq 0$ such that $y=T^n(x)$. We will write $x<_Ty$ whenever $x\neq y$ and $x\leq _Ty$. 
		Define
		$f\colon \mathcal R_T\to \N$ by: 
		\[f(x,y)\coloneqq
		\left\{
		\begin{array}{cl}
			1&\text{ if }x\leq_T y<_T U(x)\text{ or }U(x) <_T y \leq_T x,\\
			0&\text{ otherwise.}
		\end{array}
		\right.
		\]
		Let us denote by $c_U$ the $T$-cocycle of $U$. By assumption, $c_U$ is integrable. Remark that $f(x,T^n(x))=1$ if and only if $ 0\leq n < c_U(x)$ or $c_U(x)<n\leq 0$, so $\sum_{n\in\Z} f(x,T^n(x))=\abs{c_U(x)}$.
		We thus have
		\[\int_X\sum_{n\in\Z} f(x,T^n(x))d\mu=\int_X \abs{c_U(x)}d\mu<+\infty.\]
		Using the mass-transport principle (Lemma \ref{lem:MassTransport}), we deduce that $$\int_X\sum_{n\in\Z} f(T^n(x),x)d\mu<+\infty,$$ in particular
		for almost every $x\in X$, the sum $\sum_{n\in\Z} f(T^n(x),x)$ is finite. 
		
		This implies that for almost every $x\in X$, there are only finitely many integers 
		$n\in\mathbb Z$ such that 
		$U^n(x)\leq _T x <_T U^{n+1}(x)$ or $U^{n+1}(x)<_Tx \leq_T U^n(x)$.
        Since the $U$-orbit of almost every point is infinite, for almost every $x$, we must have that either $\lim_{n\to+\infty} c_{U^n}(x)=+\infty$  or $\lim_{n\to+\infty}c_{U^n}(x)=-\infty$. By ergodicity of $U$ and up to replacing $U$ with its inverse, we can assume that for almost all $x\in X$ we have $\lim_{n\to+\infty} c_{U^n}(x)=+\infty$, in particular for all but finitely many $n\geq 0$, we have that $U^n(x)\geq_T x$.
		
		By the symmetric argument and the fact that $T$ and $U$ have the same orbits, for almost all $x\in X$ and for all but finitely many $n\leq 0$, we have that $U^n(x)\leq_T x$ and therefore we must have that  
		\(\{T^n(x)\colon n\geq 0\}\bigtriangleup \{U^n(x)\colon n\geq 0\}\) is finite. The conclusion now follows from Theorem \ref{thm: belinskaya commensurates}.
	\end{proof}
	
	\bibliographystyle{alpha}
	\bibliography{bib}
	
	\Addresses
	
\end{document}